\documentclass[12pt]{amsart}
\usepackage{amssymb}

\usepackage[latin1]{inputenc}

\usepackage[all]{xy}
\usepackage{cases}
\usepackage{hyperref}

\usepackage{enumerate} \usepackage{float} 
\newcommand{\Z}{{\mathbb Z}} \newcommand{\Q}{{\mathbb Q}}

 \newcommand{\F}{{\mathbb F}}

\newcommand{\Hom}{\operatorname{Hom}\nolimits}

\newcommand{\Aut}{\operatorname{Aut}\nolimits}

\newcommand{\End}{\operatorname{End}\nolimits}

\newcommand{\tot}{\operatorname{Tot}\nolimits}
\newcommand{\im}{\operatorname{Im}\nolimits}

\newcommand{\rk}{\operatorname{rk}\nolimits}
\renewcommand{\dim}{\operatorname{rk}\nolimits}
\newcommand{\GL}{\operatorname{GL}\nolimits}

\newcommand{\sgn}{\operatorname{sgn}\nolimits}
\newcommand{\rad}{\operatorname{Rad}\nolimits}
\newcommand{\red}{\operatorname{red}\nolimits}

\newcommand{\A}{\ifmmode{\mathcal{A}}\else${\mathcal{A}}$\fi}
\newcommand{\B}{\ifmmode{\mathcal{B}}\else${\mathcal{B}}$\fi}
\newcommand{\C}{\ifmmode{\mathcal{C}}\else${\mathcal{C}}$\fi}
\newcommand{\D}{\ifmmode{\mathcal{D}}\else${\mathcal{D}}$\fi}
\newcommand{\G}{\ifmmode{\mathcal{G}}\else${\mathcal{G}}$\fi}
\newcommand{\I}{\ifmmode{\mathcal{I}}\else${\mathcal{I}}$\fi}
\newcommand{\J}{\ifmmode{\mathcal{J}}\else${\mathcal{J}}$\fi}
\newcommand{\K}{\ifmmode{\mathcal{K}}\else${\mathcal{K}}$\fi}
\renewcommand{\O}{\ifmmode{\mathcal{O}}\else${\mathcal{O}}$\fi}
\renewcommand{\P}{\ifmmode{\mathcal{P}}\else${\mathcal{P}}$\fi}
\newcommand{\U}{\ifmmode{\mathcal{U}}\else${\mathcal{U}}$\fi}
\newcommand{\M}{\ifmmode{\mathcal{M}}\else${\mathcal{M}}$\fi}
\newcommand{\N}{\ifmmode{\mathcal{N}}\else${\mathcal{N}}$\fi}
\newcommand{\Ss}{\ifmmode{\mathcal{S}}\else${\mathcal{S}}$\fi}
\newcommand{\T}{\ifmmode{\mathcal{T}}\else${\mathcal{T}}$\fi}
\newcommand{\Ff}{\ifmmode{\mathcal{F}}\else${\mathcal{F}}$\fi}
\newcommand{\Ll}{\ifmmode{\mathcal{L}}\else${\mathcal{L}}$\fi}
\newtheorem{Thm}{Theorem}[section]
\newtheorem{Conj}{Conjecture}[section]
\newtheorem{Prop}[Thm]{Proposition}
\newtheorem{Cor}[Thm]{Corollary}
\newtheorem{Lem}[Thm]{Lemma}

\theoremstyle{definition}
\newtheorem{Defi}[Thm]{Definition}
\newtheorem{Rmk}[Thm]{Remark}

\theoremstyle{remark}

\usepackage{color}

\theoremstyle{plain}

\title{Cohomology of uniserial $p$-adic space groups}

\author{Antonio D\'iaz Ramos}
\address{Departamento de \'Algebra, Geometr\'ia y Topolog\'ia,
Universidad de M\'alaga, Apdo correos 59, 29080 M\'alaga, Spain}
\email{adiazramos@uma.es}

\author{Oihana Garaialde Oca\~{n}a}
\address{Matematika Saila,
Euskal Herriko Unibertsitatearen Zientzia eta Teknologia Fakultatea,
 posta-kutxa 644, 48080 Bilbo, Spain
}
\email{oihana.garayalde@ehu.es}
\author{Jon Gonz\'alez-S\'anchez}
\address{Departamento de Matemáticas, 
Facultad de Ciencia y Tecnología de la Universidad del País Vasco,
 Apdo correos 644, 48080 Bilbao, Spain}
\email{jon.gonzalez@ehu.es}

\date{\today}

\begin{document}

\maketitle

\begin{abstract}
A decade ago, J.F. Carlson proved that there are finitely many cohomology rings of finite $2$-groups of fixed coclass, and he conjectured that this result ought to be true for odd primes \cite{Carlson}. In this paper, we prove the non-twisted case of Carlson's conjecture for any prime and we show how to proceed in the twisted case.
\end{abstract}

\section{Introduction}
In \cite{Carlson} Carlson proved the astonishing result that there are finitely many isomorphism types of cohomology rings with coefficients in $\F_2$ of $2$-groups of a fixed coclass. He also conjectured that an analogous result should hold for odd primes $p$. In this paper, we study this conjecture of Carlson, we solve it in the non-twisted case and we reduce the twisted case to a controllable situation. In order to present the results, we start recalling Leedham-Green's coclass classification of $p$-groups. A $p$-group $G$ of size $p^n$ and nilpotency class $m$ has coclass $c=n-m$. The main result in \cite{LeedGreen} states that there exist an integer $f(p,c)$ and a normal subgroup $N$ of $G$ with $|N|\leq f(p,c)$ such that $G/N$ is \emph{constructible}: we recall the definition of a constructible group in Subsection \ref{subsection:uniserialpadicspacegroupsandconstructible} below. These constructible groups arise from uniserial $p$-adic space groups and come into two flavors, either \emph{twisted} or \emph{non-twisted}, according to whether the defining twisting homomorphism is non-zero or zero respectively, see Remark \ref{rmk:nontwistedcase}. We say that $G$ is \emph{non-twisted} if for some normal subgroup $N$ of bounded size as above, the quotient $G/N$ is constructible non-twisted. Otherwise, we say that $G$ is \emph{twisted}. In the former case, $G/N$ has a large abelian normal subgroup, and in the latter case, $G/N$ has a large normal subgroup of nilpotency class $2$. 

\begin{Thm}\label{thm:intronontwist}
Let $p$ be any prime. Then there are finitely many isomorphism types of cohomology rings with coefficients in $\F_p$ of non-twisted $p$-groups of a fixed coclass.
\end{Thm}

For $p=2$, there are no twists in Leedham-Green's classification and hence every $2$-group is non-twisted. In particular, Theorem \ref{thm:intronontwist} includes Carlson's aforementioned result on $2$-groups. We reach Theorem \ref{thm:intronontwist}  in several steps starting from  abelian $p$-groups. Next, we give  a coarse description of the main milestones along this route. 

The first step is to consider $K$ and $K'$ abelian $p$-groups of the same rank $d$. Then their cohomology rings with coefficients in $\F_p$ are abstractly isomorphic regardless of the exponents (but for $p=2$ and exponent $1$). For small rank we can realize this isomorphism at the level of cochain complexes.

\begin{Prop}\label{prop:introabeliansamecohomology}
The ring isomorphism $H^*(K;\F_p)\cong H^*(K';\F_p)$ can be realized by a zig-zag of quasi-isomorphisms in the category of cochain complexes when $d<p$.
\end{Prop}

In the next step, we let an arbitrary $p$-group $P$ act on these abelian $p$-groups, and hence we must be careful enough to make the previous quasi-isomorphisms $P$-invariant. When the action may be lifted to integral matrices (Definition \ref{defi:integrallifting}) we still have control on the number of isomorphism types for semidirect products.

\begin{Prop}\label{prop:introcommonintegralliftingsamecohomology}
Let $p$ be a prime and let $\{G_i=K_i\rtimes P\}_{i\in I}$ be a family of groups such that $K_i$ is abelian of fixed rank $d<p$ for all $i$ and that all actions of $P$ have a common integral lifting. Then there are finitely many isomorphism types of rings in the collection of cohomology rings $\{H^*(G_i;\F_p)\}_{i\in I}$. 
\end{Prop}

We also show that all these rings are isomorphic as graded $\F_p$-modules. The family of maximal nilpotency class $p$-groups $\{C_{p^i}\times \ldots\times C_{p^i}\rtimes C_p\}_{\i\geq 1}$ forms a family to which the proposition applies with $d=p-1$. We surpass the bounded rank restriction by considering $n$-fold direct products.

\begin{Prop}\label{prop:introcommonintegralliftingsamecohomologyunbounded}
Let $p$ be a prime and let $\{G_i=L_i\rtimes Q\}_{i\in I}$ be a family of groups such that $L_i$ is an $n$-fold direct product, $L_i=K_i\times \ldots\times K_i$, $K_i$ is abelian of fixed rank $d<p$, $Q\leq P\wr S$ with $S\leq \Sigma_n$ and all actions of $P$ on $K_i$ have a common integral lifting. Then there are finitely many isomorphism types of rings in the collection of cohomology rings $\{H^*(G_i;\F_p)\}_{i\in I}$. 
\end{Prop}

Again, we also show here that all these rings are isomorphic as graded $\F_p$-modules. Next we move to uniserial $p$-adic space groups, as they give rise to constructible groups. So let  $R$ be a uniserial $p$-adic space group with translation group $T$ and point group $P$, and let $T_0$ be the minimal $P$-lattice  for which the extension $T_0\to R_0\to P$ splits (Subsection \ref{subsection:uniserialpadicspacegroupsandconstructible}). We prove the next result either using Proposition \ref{prop:introcommonintegralliftingsamecohomology} and Nakaoka's theorem on wreath products or more directly by employing Proposition \ref{prop:introcommonintegralliftingsamecohomologyunbounded}. The connection with the aforementioned integral liftings is that, up to conjugation, $P$ may be chosen to act by integral matrices on the lattice $T$. The relation to the symmetric group is that the standard uniserial $p$-adic space group has a point group that involves $\Sigma_n$.

\begin{Prop}\label{prop:introspacegroupssamecohomology}
There are finitely many isomorphism types of rings for the cohomology rings $H^*(T_0/U\rtimes P;\F_p)$ for the infinitely many $P$-invariant lattices $U<T$.
\end{Prop}

From here, we prove Theorem \ref{thm:intronontwist} by using certain refinements of Carlson's counting arguments for spectral sequences (Section \ref{section:countingtheorems}), Leedham-Green's classification and a detailed description of constructible groups. An immediate consequence of Proposition \ref{prop:introspacegroupssamecohomology} is the following result, in which we have dropped the splitting condition.

\begin{Prop}\label{prop:introspacegroupssamecohomologynonsplit}
There are finitely many possibilities for the ring structure of the graded $\F_p$-module $H^*(R/U;\F_p)$ for the infinitely many $P$-invariant lattices $U<T$. 
\end{Prop}

 Regarding the twisted case of Carlson's conjecture, we reduce it to a similar problem about realizations of abstract isomorphism between cohomology rings of certain $p$-groups. In this case, we are concerned about an abelian $p$-group $A$ and its twisted version $A_\lambda$ (Subsection \ref{subsection:twistedabelianpgroups}), and again we consider invariance under the action of certain $p$-group $P$. Under mild assumptions, that hold in the context of the Leedham-Green classification, the cohomology rings with coefficients in $\F_p$ of $A$ and $A_\lambda$ are abstractly isomorphic (Subsection \ref{subsection:cohotwistedabelianpgroup}).

\begin{Conj}\label{conj:intro}
The abstract isomorphism $H^*(A;\F_p)\cong H^*(A_\lambda;\F_p)$ can be realized in the category of cochain complexes via a zig-zag of $P$-invariant quasi-isomorphisms.
\end{Conj}

We prove that Conjecture \ref{conj:intro} implies the twisted case of Carlson's conjecture, and hence also Carlson's conjecture in full generality.

\begin{Thm}\label{thm:introgeneral}
If Conjecture \ref{conj:intro} holds, then for any prime $p$ there are finitely many isomorphism types of cohomology rings with coefficients in $\F_p$ of $p$-groups of a fixed coclass.
\end{Thm}

Explicit isomorphisms among cohomology rings of certain families of $p$-groups are expected, see \cite[Question 6.1]{Carlson}. In \cite[Conjecture 3]{BettinaDavid2015}, Eick and Green  refine this question via coclass families, and prove that it holds asymptotically up to $F$-isomorphism. In our setting, we expect that there is just one isomorphism type for the collections in Propositions \ref{prop:introcommonintegralliftingsamecohomology} and \ref{prop:introcommonintegralliftingsamecohomologyunbounded}. Forgetting the ring structure, G. Ellis has proven in \cite{Ellis2016} by other methods that, for a uniserial $p$-adic space group $R$ with translation group $T$ of  dimension $d$, $H_n(R/T_i;\F_p)$ and $H_n(R/T_{i+d};\F_p)$ are isomorphic $\F_p$-vector spaces for $i$ big enough and any $n$. Here, $\{T_i\}_{i\geq 0}$ is the \emph{uniserial filtration} of $R$, see Subsection \ref{subsection:uniserialpadicspacegroupsandconstructible}. We expect that the cohomology rings $H^*(R/T_i;\F_p)$ and $H^*(R/T_{i+d};\F_p)$ are isomorphic, and Proposition \ref{prop:introspacegroupssamecohomologynonsplit} already points in that direction. Here it is an exhaustive account of the layout of this work.
\begin{enumerate}[{Section} 1:]
 \setcounter{enumi}{1}
\item We fix notation for the homological algebra objects that we use (Subsection \ref{subsection:notationsandsignconventions}), including the standard resolution (Subsection \ref{subsection:standardresolution}) and resolutions for semidirect products (Subsection \ref{subsection:resolution-semidirect-product}). We also prove a lemma (Subsection \ref{subsection:somehomalg}) that is fundamental to deal with cohomology rings of  semidirect products, and hence for Propositions \ref{prop:introcommonintegralliftingsamecohomology} and \ref{prop:introcommonintegralliftingsamecohomologyunbounded}.
\item Here we summarise results about uniserial $p$-adic space groups and Leedham-Green's constructible groups (Subsection \ref{subsection:uniserialpadicspacegroupsandconstructible}), twisted abelian $p$-groups (Subsections \ref{subsection:twistedabelianpgroups} and \ref{subsection:furtherpropertiesoftwistedabelian}), powerful $p$-central groups (Subsection \ref{subsection:pcentralpowerfulomegaextprop}) and standard  uniserial $p$-adic space groups (Subsection \ref{subsection:standarduniserialpadicspacegroups}). We also introduce split constructible groups and give a very detailed description of them via twisted abelian $p$-groups (Lemma \ref{lemma:GalphaGalpha0}).
\item We prove certain refinements of Carlson's counting arguments in \cite{Carlson}.
\item This section is devoted to cohomology of abelian $p$-groups of small rank (Subsection \ref{subsection:cohoabeliansmallrank}), their semidirect products (Subsection \ref{subsection:cohomologysmallrank}) and the unbounded rank case (Subsection \ref{subsection:cohomologyunboundedrank}). We prove Proposition \ref{prop:introabeliansamecohomology}, \ref{prop:introcommonintegralliftingsamecohomology} and \ref{prop:introcommonintegralliftingsamecohomologyunbounded}  as Corollary \ref{cor:abeliansamecohomology} and Propositions \ref{prop:commonintegralliftingsamecohomology} and \ref{prop:commonintegralliftingsamecohomologyunbounded} respectively.
\item Cohomology of uniserial $p$-adic groups is addressed (Subsections \ref{subsection:cohostandardpadicspacegroup} and \ref{subsection:cohopadicspacegroups}) and Propositions \ref{prop:introspacegroupssamecohomology} and \ref{prop:introspacegroupssamecohomologynonsplit}  are proven as Proposition \ref{prop:spacegroupssamecohomology}  and Corollary \ref{cor:spacegroupssamecohomologynonsplit}  respectively (see also Propositions \ref{prop:standardspacegroupsamecohomology} and  \ref{prop:spacegroupssamebigradedalgebras}). Moreover, cohomology of twisted abelian $p$-groups is studied (Subsection \ref{subsection:cohotwistedabelianpgroup}), a fully detailed version of Conjecture \ref{conj:intro} is given as Conjecture \ref{conj:detailed} and some of its consequences are proven.
\item In Theorem \ref{thm:Carslonconj}, we give a proof that Theorem \ref{thm:introgeneral} holds assuming Conjecture \ref{conj:intro}/\ref{conj:detailed}. The same proof shows that the non-twisted case Theorem \ref{thm:intronontwist} holds with no assumptions.
\end{enumerate}

\noindent \textbf{Acknowledgements:} We would like to thank Alex Gonz\'alez, Lenny Taelman and Antonio Viruel for some discussions and advice related to this work. We are also really grateful to the referee for performing a meticulous reading and for all his/her comments.

Antonio Díaz Ramos is supported by MICINN grant RYC-2010-05663 and partially supported by MEC grant MTM2013-41768-P and Junta de Andaluc{\'\i}a grant FQM-213. Oihana Garaialde Ocaña acknowledges the  support by grants MTM2011-28229-C02-02 and  MTM2014-53810-C2-2-P, from the Spanish Ministry of Economy and Competitivity and by the Basque Government, grants IT753-13, IT974-16 and Ph.D. grant PRE$\underline{\ }2015\underline{\ }2\underline{\ }0130$. Jon Gonzalez-Sanchez acknowledges the  support by grants MTM2011-28229-C02-01 and  MTM2014-53810-C2-2-P, from the Spanish Ministry of Economy and Competitivity,  the Ramon y Cajal Programme of the Spanish Ministry of Science and Innovation, grant RYC-2011-08885, and by the Basque Government, grants IT753-13 and IT974-16.

\section{Preliminaries: Homological algebra}
\label{section:preliminaries}

Throughout the paper we denote by $k$ a field and by $G$ a (possibly infinite) discrete group. We also denote the group algebra of $G$ with coefficients in $k$ by $kG$. We shall deal with (bounded below) single complexes and (first quadrant) double complexes in the category of $kG$-modules. We define a \emph{quasi-isomorphism} as a map of cochain complexes which induce a $k$-isomorphism in cohomology in all degrees. Finally, for a prime $p$, we shall often use the cyclic group with $p^i$ elements, which we denote by $C_{p^i}$. If $p$ is even we further assume that $i\geq 2$. All the material presented in this section but subsection \ref{subsection:somehomalg} is standard and can be found in \cite{KSBrown82}, \cite{LEvens91}, \cite{CAWeibel94} and \cite{May67}.

\subsection{Notations and sign convention}\label{subsection:notationsandsignconventions} If $A_*$ is a chain complex of $kG$-modules with differential $d_A$ and $M$ is a $kG$-module, then $\Hom_{kG}(A_*,M)$ becomes a cochain complex with $d(f)(a)=(-1)^{n+1}f(d_A(a))$, where $f\in \Hom_{kG}(A_n,M)$ and $a\in A_{n+1}$. If $A_*$ and $B_*$ are chain or cochain complexes of $kG$-modules with differentials $d_A$ and $d_B$ we denote by $C=A\otimes_{kG} B$ the double complex $C_{n,m}=A_n\otimes_{kG} B_m$ with differentials 
\[
d_h(a\otimes b)=d_A(a)\otimes b\text{ and }d_v(a\otimes b)=(-1)^n a\otimes d_B(b)\textit{, for $a\otimes b\in A_n\otimes_{kG} B_m$}.
\] 
Similarly, let $A_*$ be a chain complex of $kG$-modules and let $B^*$ be a cochain complex of $kG$-modules with differentials $d_A$ and $d_B$. Denote by $D=\Hom_{kG}(A,B)$ the double complex $D^{n,m}=\Hom_{kG}(A_n, B^m)$. The differentials  are given in this case by 
\[
d^h(f)(a)=(-1)^{n+m+1}f(d_A(a))\text{ and }d^v(f)(a)=d_B(f(a))\textit{, for $f\in \Hom_{kG}(A_n, B^m)$.}
\]
As usual, the total complex $\tot(C)$ is a chain complex with differential $d_h+d_v$ and $\tot(D)$ is a cochain complex with differential $d^h+d^v$.

Assume now $N\unlhd G$, that $A_*$ is a chain complex of $k(G/N)$-modules and that $B_*$ is a chain complex of $kG$-modules. Then for each $n$ and $m$ and each $kG$-module $M$ we have an isomorphism of $k$-modules
\[
\Hom_{kG}(A_n\otimes_k B_m,M)\cong \Hom_{k(G/N)}(A_n,\Hom_{kN}(B_m,M)),
\]
where $G$ acts diagonally on $A_n\otimes_k B_m$ and $G/N$ acts on $f\in \Hom_{kN}(B_m,M)$ via $(\overline g\cdot f)(b)=g\cdot f(g^{-1}\cdot b)$ for $g\in G$ and $b\in B_m$. Consider the double complexes $C=A\otimes_k B$ and $D=\Hom_{kG/N}(A,\Hom_{kN}(B,M))$. Then, via the isomorphism above and with the sign conventions described, the two cochain complexes $\Hom_{kG}(\tot(C),M))$ and $\tot(D)$ are identical.

Finally, by a \emph{product} on a cochain complex $B^*$ with differential $d_B$, we mean a degree preserving $k$-bilinear form
\[
\cup\colon  \tot(B^*\otimes B^*)\to B^*
\]
such that $\cup$ is associative with unit and satisfies the Leibnitz rule, i.e.:
\[
d_B(b\cup b')=d_B(b)\cup b'+(-1)^n b\cup d_B(b'), 
\]
for $b\in B^n$ and $b'\in B^{n'}$. If $B^*$ has a product we say that $B^*$ is a \emph{differential graded algebra}. For instance, if $A^*$ and $B^*$ are differential graded algebras, then $C=A\otimes_k B$ is also a differential graded algebra by defining:
\begin{equation}\label{equ:productforotimes}
(a\otimes b)\cup (a'\otimes b')=(-1)^{n'm}(a\cup_A a')\otimes (b\cup_B b')
\end{equation}
for $b\in B^m$ and $a'\in A^{n'}$. For the standard resolution of a group $G$,  $B_*(G;k)$ (defined in Subsection \ref{subsection:standardresolution}), and a differential graded algebra of $kG$-modules $K^*$, we may define a product $\cup$ on the total complex $\tot(D)$ for $D_{n,m}=\Hom_{KG}(B_n(G;K),K^m)$. For $f_1\otimes f_2\in \Hom_{kG}(B_{n_1}(G;k),K^{m_1})\otimes \Hom_{kG}(B_{n_2}(G;k),K^{m_2})$ define $f_1\cup f_2\in \Hom_{kG}(B_{n_1+n_2}(G;k),K^{m_1+m_2})$ as the function that on $(g_0,\ldots,g_{n_1+n_2})$ evaluates to
\begin{equation}\label{equ:productforHom}
(-1)^{n_1(n_2+m_2)}f_1(g_0,\ldots,g_{n_1})\cup_K f_2(g_{n_1},\ldots,g_{n_1+n_2}),
\end{equation}
where the sign implements the Koszul sign rule. This last construction is more general (see \cite[p.137]{May67}) but we need only this simple version here.

\subsection{Standard resolution and standard cup product}\label{subsection:standardresolution}\label{subsection:standardcupproduct}
We recall the standard resolution $B_*(G;k)$ of the trivial $kG$-module $k$: $B_n(G;k)$ is the free $k$-module $kG^{n+1}$  with diagonal $G$-action
$$
g\cdot(g_0,\ldots,g_n)=(gg_0,\ldots, gg_n).
$$
The differential $\delta_n\colon B_n(G;k)\to B_{n-1}(G;k)$ is the alternate sum $\sum_{i=0}^n (-1)^i\partial_i$, where $\partial_i$ is the $i^{th}$-face map $B_n(G;k)\to B_{n-1}(G;k)$ with 
$$
\partial_i(g_0,\ldots,g_n)=(g_0,\ldots,\widehat g_i,\ldots,g_n).
$$
The augmentation $\epsilon\colon B_0(G;k)\to k$ sends $g_0\mapsto 1$ for all $g_0\in G$. Now let $M$ be any $kG$-module. Then $C^*(G;M)=\Hom_{kG}(B_*(G;k),M)$ is a cochain complex of $k$-modules whose differential $\delta^n\colon C^n(G,k)\to C^{n+1}(G,k)$ is given as in Subsection \ref{subsection:notationsandsignconventions}, i.e., by $\delta^n=(-1)^{n+1}\Hom_{kG}(\delta_{n+1},M)$. Its cohomology is $H^*(G;M)$.

The cup product at the cochain level for the standard resolution,
\[ 
\cup\colon C^*(G;k) \otimes C^*(G;k)\to C^*(G;k),
\]
takes $f_1\otimes f_2\in C^{n_1}(G;k)\otimes C^{n_2}(G;k)$ to $f_1\cup f_2\in C^{n_1+n_2}(G,k)$ defined by 
\begin{equation}\label{equ:standardrescupproduct}
(f_1\cup f_2)(g_0,\ldots,g_{n_1+n_2})=(-1)^{n_1n_2}f_1(g_0,\ldots,g_{n_1})f_2(g_{n_1},\ldots,g_{n_1+n_2}),
\end{equation}
where the sign again implements the Koszul sign rule \cite[p.110]{KSBrown82}. It is well known that $(C^*(G;k),d,\cup)$ is a differential graded $k$-algebra (associative with unit). This product induces in $H^*(G;k)$ the usual cup product, which we denote by the same symbol ($[f_1]\cup [f_2]=[f_1\cup f_2]$). 

\subsection{Resolutions and cup products for semidirect products}\label{subsection:resolution-semidirect-product}

Let $G=N\rtimes P$ be a semidirect product and denote by conjugation $p\cdot n ={} ^p n=pnp^{-1}$ the action of $P$ on $N$. Assume that $N_*$ is a $kN$-resolution of $k$, $P_*$ is a $kP$-resolution of $k$ and that $P$ acts on $N_*$ in such a way that:
\begin{enumerate}
\item The action of $P$ commutes with the augmentation and the differentials of $N_*$.
\item For all $p\in P$, $n\in N$ and $z\in N_*$, $p\cdot (n\cdot z)=(p\cdot n)\cdot (p\cdot z)$ .
\end{enumerate}
Consider the double complex $C=C_{*,*}=P_*\otimes N_*$. 
Then there is an action of $G=N\rtimes P$ over $C$ and over $\tot(C)$ described by
$$
(n,p)\cdot(z_P \otimes z_N)=p\cdot z_P \otimes n\cdot(p \cdot z_N),
$$
where $n \in N,$ $p\in P,$ $z_P \in P_*$ and $z_N \in N_*$. Moreover, following \cite[p. 19]{LEvens91}, $\tot(C)$ is a $kG$-projective resolution of the trivial $kG$-module $k$ and thus, for any $kG$-module $M$, we have $H^*(G;M)\cong H^*(\Hom_{kG}(\tot(C),M))$ as graded $k$-modules. 

As particular case, consider $P_*=B_*(P;k)$ and $N_*=B_*(N;k)$. Then the action of $P$ on $B_*(N;k)$ given by 
\[
p \cdot (n_0, \dots, n_m)=({}^p n_0, \dots,{}^p n_m)
\]
satisfies (1) and (2) above. So $H^*(G;M)\cong H^*(\Hom_{kG}(\tot(C),M))$ as graded $k$-modules for $C=C_{*,*}=B_*(P;k)\otimes B_*(N;k)$. In fact, we can endow $\Hom_{kG}(\tot(C),k)$ with a product as in Equation \eqref{equ:productforHom}: 
\begin{scriptsize}
\[
\Hom_{kP}(B_{n_1}(P;k),C^{m_1}(N;k))\otimes \Hom_{kP}(B_{n_2}(P;k),C^{m_2}(N;k))\to \Hom_{kP}(B_{n_1+n_2}(P;k),C^{m_1+m_2}(N;k))
\] 
\end{scriptsize}sends $f_1\otimes f_2$ to the map that on $(p_0,\ldots,p_{n_1+n_2})$ evaluates to the function that on $(n_0,\ldots,n_{m_1+m_2})$ takes the following value:
\begin{small}
\[
(-1)^{n_1(n_2+m_2)+m_1m_2}f_1(p_0,\ldots,p_{n_1})(n_0,\ldots,n_{m_1})f_2(p_{n_1},\ldots,p_{n_1+n_2})(n_{m_1},\ldots,n_{m_1+m_2}),
\]
\end{small}where again the sign reflects the Koszul sign rule. That this product induces the usual cup product in $H^*(G;k)$ is  a consequence of the criterion \cite[XII.10.4]{SMacLane63}.

\subsection{A lemma for spectral sequences}\label{subsection:somehomalg}

In this subsection we prove the following easy and useful result:

\begin{Lem}\label{lemma:homalg}
Let $G$ be a group, let $K^*$ and ${K'}^*$ be cochain complexes of $kG$-modules and let $P_*$ be a chain complex of free $kG$-modules. If $\varphi\colon K^*\to {K'}^*$ is a quasi-isomorphism then 
\[
\tot(\Hom_{kG}(P_*,K^*))\stackrel{\tot(\varphi_*))}\longrightarrow \tot(\Hom_{kG}(P_*,{K'}^*))
\]
is a quasi-isomorphism  and so $H^*(\tot(\Hom_{kG}(P_*,K^*))\cong H^*(\tot(\Hom_{kG}(P_*,{K'}^*))$ as graded $k$-modules.
\end{Lem}

\begin{proof} Consider the double complexes of $k$-modules $C$ and $C'$ given by  $C_{n,m}=\Hom_{kG}(P_n,K^m)$ and $C'_{n,m}=\Hom_{kG}(P_n,{K'}^m)$. It is clear that the cochain map $\varphi$ induces a morphism of double complexes $\varphi_*\colon C_{*,*} \rightarrow C'_{*,*}$ given by $f \mapsto \varphi \circ f$. Consider the total complexes $\tot(C)$ and $\tot(C')$. The filtrations by columns produce spectral sequences $E$ and $E'$ converging to the cohomology of $\tot(C)$ and $\tot(C')$ respectively.  The first page $E_1$ of $E$  is obtained by taking  cohomology with respect to the vertical differential:
\[
{E_1}^{n,m}= H^m(\Hom_{kG}(P_n, K^*))=\Hom_{kG}(P_n, H^m(K^*)),
\]
where $\Hom_{kG}(P_n, \cdot)$ commutes with the cohomology functor because $P_n$ is a free $kG$-module and hence $\Hom_{kG}(P_n, \cdot)$ is an exact functor. As the differential of $K^*$ commutes with the $G$-action, $H^*(K)$ becomes a $kG$-module and $\Hom_{kG}(P_*, H^*(K))$ is well-defined. The morphism of double complexes $\varphi_*$ induces a morphism of spectral sequences $\Phi: E \rightarrow E'$. Between the first pages, the morphism 
$$
\Phi_1:\Hom_{kG}(P_*, H^*(K)) \rightarrow \Hom_{kG}(P_*, H^*(K'))
$$ 
is given by post-composing with $H^*(\varphi)$, that is, $f \mapsto H^*(\varphi) \circ f$. This  is an isomorphism as $H^*(\varphi)$ is an isomorphism by hypothesis. Therefore, all morphisms $\Phi_r\colon E_r\rightarrow E'_r$ are isomorphisms for $r\geq 1$ and 
$$
H^*(\tot(\varphi_*))\colon H^*(\tot(C))\cong H^*(\tot(C'))
$$
is an isomorphism too. 
\end{proof}

We shall need the following version  of the previous lemma that involves products:

\begin{Lem}\label{lemma:homalgalgebra}
With the hypothesis of Lemma \ref{lemma:homalg}, suppose in addition that:
\begin{enumerate}
\item $P_*=B_*(G;k)$ is the standard resolution of $G$.
\item $K^*$ and $K'^*$ are equipped with products $\cup$ and $\cup'$.
\item $H^*(\varphi)\colon H^*(K^*)\to H^*(K'^*)$ preserves the induced products.
\end{enumerate} 
Then there are filtrations of the graded algebra $H^*(\tot(\Hom_{kG}(P_*,K^*)))$ and of the graded algebra $H^*(\tot(\Hom_{kG}(P_*,{K'}^*)))$ such that the associated bigraded algebras are isomorphic. 
\end{Lem}
\begin{proof}
Consider the double complexes $C_{n,m}=\Hom_{kG}(B_n(G;k),K^m)$ and $C'_{n,m}=\Hom_{kG}(B_n(G;K),{K'}^m)$. By \eqref{equ:productforHom}, there are products $\cup_C$ and $\cup_{C'}$ on $\tot(C)$ and $\tot(C')$ that induce products in $H^*(\tot(C))$ and $H^*(\tot(C'))$ respectively. The products $\cup_C$ and $\cup_{C'}$ are defined in such a way that the filtrations by columns preserve them. Hence, the associated spectral sequences are spectral sequences of algebras and the bigraded algebra structures on $E_\infty$ and $E'_\infty$ arise from the the induced filtrations of $H^*(\tot(C))$ and $H^*(\tot(C'))$ respectively. The morphism $\Phi_1\colon E_1\to E'_1$ is given by post-composition by $H^*(\varphi)$, and hence it preserves the bigraded algebra structures on $E_1$ and $E'_1$ because of hypothesis ($3$). From here, it is easy to see that $\Phi_r\colon E_r\to E'_r$ preserve the bigraded algebra structures for all $r\geq 1$, in particular, for $r=\infty$. This proves the claim.
\end{proof}

\begin{Rmk}
Note that we do not claim that the two graded $k$-algebras 
\[
H^*(\tot(\Hom_{kG}(P_*,K^*)))\text{ and }H^*(\tot(\Hom_{kG}(P_*,{K'}^*)))
\]
are isomorphic.
\end{Rmk}

\section{Preliminaries: Uniserial $p$-adic space groups}
\label{section:uniserialpadicspacegroups}
In this section, we state general facts about uniserial $p$-adic space groups (see \cite{LeedGreenBook}, \cite{McKay94} or \cite{Eick2005}), about powerful $p$-central groups with the $\Omega$-extension property (see \cite{Weigel}), about twisted abelian $p$-groups, and about constructible groups and the Leedham-Green classification of finite $p$-groups of fixed coclass (see \cite{LeedGreen}).

\subsection{Uniserial $p$-adic space groups and constructible groups}
\label{subsection:uniserialpadicspacegroupsandconstructible}
Let $R$ be a uniserial $p$-adic space group with translation group $T$ and point group $P$. Then $R$ is a $p$-adic pro-$p$ group that fits in the extension of groups
\begin{equation}\label{equ:TRP}
1 \to T \to R \to P \to 1,
\end{equation}
and $T$ is the maximal normal abelian pro-$p$ subgroup of $R$.  The translation group is a $\Z_p$-lattice of rank $d_x=p^{x-1}(p-1)$ and $P$ is a finite $p$-group which acts faithfully and uniserially on $T$. From the former condition, we have that $P \leq \GL_{d_x}(\Z_p)$. From the latter condition, the  group $R$ has finite coclass $c$ with $x\leq c$ and there exists a unique series of $P$-invariant lattices, the \emph{uniserial filtration}. More precisely, for each integer $i\geq 0$, there exists a unique $P$-invariant sublattice $T_i$ of $T$ that satisfies $|T:T_i|=p^i$. Moreover, for $i=j+sd_x$ with $s\geq 0$, we have $T_i=p^sT_j$. For $\tilde{T}:= T \otimes_{\Z_p} \Q_p$, we obtain a split extension of $\tilde{T}$ by the point group $P$. By Lemma 10.4.3 in \cite{LeedGreenBook}, there exists a minimal superlattice $T_0$ of $T$ in $\tilde{T}$ such that the subgroup $R_0$ generated by $T_0$ and $P$ splits over $T_0$:
\begin{equation}\label{equ:T0R0P}
1 \to T_0 \to R_0 \to P \to 1.
\end{equation}
Notice that the index of $T$ in $T_0$ is finite, and therefore, the index of $R$ in $R_0$ is also finite. The following result of Leedham-Green describes the structure of finite $p$-groups of fixed coclass.

\begin{Thm}[{Leedham-Green \cite[Theorem 7.6, Theorem 7.7]{LeedGreen}}]
\label{thm:quotientofcoclasscareconstructible}
For some function $f(p,c)$, almost every $p$-group $G$ of coclass $c$ has a normal subgroup $N$ of order at most $f(p,c)$ such that $G/N$ is constructible. 
\end{Thm}

A constructible group arises from the following data: A uniserial $p$-adic space group $R$ with translation group $T$ and point group $P$, two $P$-invariant sublattices of $T$, $U<V$, and $\gamma \in \Hom_P(\Lambda^2(T_0/V),V/U)$. In the next paragraphs, we review the definition of the \emph{constructible group} $G_\gamma$ associated to $R$, $U$, $V$ and $\gamma$ by Leedham-Green \cite[p. 60]{LeedGreen}. At the same time, we consider a related group $G_{\gamma,0}$ that we term \emph{split constructible group}.

\begin{Rmk}
For $p=2$, we assume, following Leedham-Green, that $\gamma=0$, i.e., there is no twist in this case.
\end{Rmk}
%
%

Consider the group $X$ with underlying set $V/U\times T_0/V$ and operation 
\[
(a_1,b_1)(a_2,b_2)=(a_1+a_2+\frac{1}{2}\gamma(b_1,b_2),b_1+b_2).
\]
Then $P$ acts on $X$ coordinate-wise and we may consider the group $X \rtimes P$ and the following extension
\begin{equation}\label{equ:XsdPextension}
1\to V/U \to X \rtimes P  \to T_0/V\rtimes P=R_0/V\to 1.
\end{equation}
Pulling back along the inclusion $R/V\leq R_0/V$, we get another extension $Y$:
\begin{equation}\label{equ:Yextension}
1\to V/U \to Y  \to R/V\to 1.
\end{equation}

On the other hand, there are extension of groups
\begin{equation}
1\to V/U \to R_0/U  \to R_0/V \to 1\label{equ:R0/Usplitextension}
\end{equation}
and
\begin{equation}
1\to V/U \to R/U  \to R/V \to 1.\label{equ:R/Usplitextension}
\end{equation}

\begin{Defi}\label{Defi:constructiblegroup}
The \emph{constructible group} associated to $R$, $U$, $V$ and $\gamma$ is the group $G_{\gamma}$ obtained as the Baer sum of the extensions \eqref{equ:Yextension} and \eqref{equ:R/Usplitextension}.
\end{Defi}

\begin{Defi}\label{Defi:splitconstructiblegroup}
The \emph{split constructible group} associated to $R$, $U$, $V$ and $\gamma$ is the group $G_{\gamma,0}$ obtained as the Baer sum of the extensions \eqref{equ:XsdPextension} and \eqref{equ:R0/Usplitextension}.
\end{Defi}

\begin{Rmk}\label{rmk:nontwistedcase}
We say that $G_\gamma$ and $G_{\gamma_,0}$ are \emph{twisted} or \emph{non-twisted} according to the conditions $\gamma\neq 0$ or $\gamma=0$ respectively. In the non-twisted case, the group $X$ is the direct product $V/U \times T_0/V$, $G_\gamma=R/U$, $G_{\gamma,0}=R_0/U$.
\end{Rmk}

%



\subsection{Twisted abelian $p$-groups.}
\label{subsection:twistedabelianpgroups} In this subsection, we describe certain $p$-groups that are obtained by twisting the sum operation in an abelian $p$-group. They play a central role in  the description of Leedham-Green constructible groups. 

\begin{Defi}\label{defi:HomAABHom_PAAB}
Let $A$ and $B$ be abelian $p$-groups. We denote by $\Hom(\Lambda^2 A,B)$ the set of maps $\lambda: A \times A \to B$ which are biadditive and alternating.
If, in addition, $P$ is a $p$-group that acts on $A$ and $B$, then we denote by $\Hom_P(\Lambda^2 A,B)$ the subset of $\Hom(\Lambda^2 A,B)$ consisting of the maps $\lambda$ that satisfy:
\[
p\cdot \lambda(a,a')=\lambda(p\cdot a,p\cdot a')\text{ for all $a,a'\in A$, all $p\in P$.}
\]
\end{Defi}

The statements in the following definitions are easy to check.

\begin{Defi}\label{defi:Alambda}
Let $A=(A,+)$ be an abelian $p$-group for an odd prime $p$ and let $\lambda \in \Hom(\Lambda^2 A,A)$. The set $A$ endowed with the operation
$$
a+_{\lambda} a' = a+a'+\tfrac{1}{2} \lambda(a,a')\text{, for $a,a' \in A$,}
$$
defines a group structure that we denote by $A_{\lambda}=(A, +_{\lambda})$. 
\end{Defi}

\begin{Rmk}\label{rmk:alwaysclasstwo}
Throughout the paper we assume that the map $\lambda$ used to define $A_\lambda$ (Definition \ref{defi:Alambda}) satisfy
\[
\im(\lambda)\leq \rad(\lambda)=\{a\in A|\lambda(a,a')=0\text{ for all $a'\in A$}\}.
\]
This ensures that $A_\lambda$ has nilpotency class at most two, see Lemma \ref{propertiesAlambda}\eqref{propertiesAlambdaclasstwo}.
\end{Rmk}

\begin{Defi}\label{defi:AlambdasdP}
With the notations of Definition \ref{defi:Alambda} and if we further assume that $P$ is a $p$-group that acts on $A$ and that $\lambda \in \Hom_P(\Lambda^2 A,A)$, then $P$ acts on $A_\lambda$ and hence we may consider the semi-direct product $A_\lambda\rtimes P$.
\end{Defi}

The group $A_\lambda$ has the following properties:
\begin{Lem} \label{propertiesAlambda}
With the notations of Definition \ref{defi:Alambda}, the following properties hold:
\begin{enumerate}[(i)]
\item \label{propertiesAlambdasamepowers}
For all $a\in A$ and all $n\in \Z$ we have
$
\overbrace{a+_\lambda+a+_\lambda\ldots+_\lambda a}^\text{$n$ times}=\overbrace{a+a+\ldots+a}^\text{$n$ times}
$.
\item \label{propertiesAlambdasameOmega1}
$\Omega_1(A)=\Omega_1(A_\lambda)$.
\item \label{propertiesAlambdaclasstwo}
$A_\lambda$ is abelian or has nilpotency class two.
\item \label{propertiesAlambdaisoonsubgroup}
If $R$ is a subset of $\rad(\lambda)$ then $R$ is a subgroup of $A$ if and only if it is a subgroup of $A_\lambda$. Moreover, in that case, the identity map $(R, +) \to (R, +_{\lambda})$ is an isomorphism. 
\item \label{propertiesAlambdaisoonquotient}
If $R$ is a subgroup satisfying $\im(\lambda)\subseteq R\subseteq \rad(\lambda)$ then $R$ is central in both $A$ and $A_\lambda$ and the identity map  $(A/R, +) \to (A_\lambda/R,+_{\lambda})$ is an isomorphism.
\end{enumerate}
\end{Lem}

\begin{proof} 
Points \eqref{propertiesAlambdasamepowers} and \eqref{propertiesAlambdaclasstwo} follow from the fact that $\lambda(a,a)=0$ for all $a\in A$ and from the condition $\im(\lambda)\subseteq \rad(\lambda)$ in Remark \ref{rmk:alwaysclasstwo} respectively. Then \eqref{propertiesAlambdasameOmega1} is a consequence of \eqref{propertiesAlambdasamepowers} and the fact that $pa=pb=0\Rightarrow p(a+_\lambda b)=0$.  The fourth item \eqref{propertiesAlambdaisoonsubgroup} is a consequence of the fact that $\lambda(a,A)=0$ for all $a\in R$. Point \eqref{propertiesAlambdaisoonquotient} is also straightforward.
\end{proof}

In the next result, we give an alternative description of the group $G_{\gamma,0}$ (see Definition \ref{Defi:splitconstructiblegroup}) in terms of a twisted abelian $p$-group. We refer the reader to the notation of Subsection \ref{subsection:uniserialpadicspacegroupsandconstructible}. We denote by $\lambda$ the element in $\Hom_P(\Lambda^2(T_0/U),T_0/U)$ obtained from $\gamma$ precomposing it with the projection $T_0/U\twoheadrightarrow T_0/V$ and postcomposing it with the inclusion $V/U\hookrightarrow T_0/U$. Then we consider the $p$-group $(T_0/U)_{\lambda}=(T_0/U, +_{\lambda})$ (Definition \ref{defi:Alambda}), and the semi-direct product $(T_0/U, +_{\lambda})\rtimes P$ (Definition \ref{defi:AlambdasdP}).
\begin{Lem}\label{lemma:GalphaGalpha0}
The constructible group $G_\gamma$ and the split constructible group $G_{\gamma,0}$ satisfy the following properties:
\begin{enumerate}
\item $G_\gamma$ is a subgroup of $G_{\gamma,0}$ of index $|T_0:T|$.
\label{lemma:GalphaGalpha0subgroupindex}
\item There is an isomorphism $G_{\gamma,0}\cong (T_0/U, +_{\lambda})\rtimes P$.
\label{lemma:GalphaGalpha0split}
\end{enumerate}
\end{Lem}
\begin{proof}
The first point is a consequence of that the short exact sequences \eqref{equ:Yextension} and \eqref{equ:R/Usplitextension} are sub-short exact sequences of \eqref{equ:XsdPextension} and \eqref{equ:R0/Usplitextension} respectively. The index is as indicated because $|R_0:R|=|T_0:T|$.

Before proving the second point of the statement, consider the following diagram, where the upper row is the extension \eqref{equ:R0/Usplitextension}:
\[
\xymatrix{
1\ar[r]&V/U\ar[r]&R_0/U=T_0/U\rtimes P\ar[r]&R_0/V=T_0/V\rtimes P\ar[r]&1\\
1\ar[r]&V/U\ar@{=}[u]\ar[r]&T_0/U\ar[u]\ar[r]&T_0/V\ar[u]\ar[r]&1.
}
\]
Choose a (set-theoretical) section $s\colon T_0/V\to T_0/U$ satisfying $s(1)=1$.  Then $s\times 1\colon R_0/V\to R_0/U$ is also a section. We define the corresponding cocycles $\theta\in C^2(T_0/V;V/U)$ and $\eta\in C^2(R_0/V;V/U)$ by
\begin{align*}
\theta(x_1,x_2)&=s(x_1)+s(x_2)-s(x_1+x_2)\text{ and }\\
\eta(x_1,p_1,x_2,p_2)&=(s\times 1)(x_1,p_1)+(s\times 1)(x_2,p_2)-(s\times 1)((x_1,p_1)+(x_2,p_2))\\
&=(s(x_1)+{}^{p_1}s(x_2)-s(x_1+{}^{p_1}x_2),1).
\end{align*}
A short calculation shows that they satisfy
\begin{equation}\label{equ:extensionclassesforinvariantsection}
\eta(x_1,p_1,x_2,p_2)=\theta(x_1,{}^{p_1}x_2)+{}^{p_1}s(x_2)-s({}^{p_1}x_2),
\end{equation}
where $(x_i,p_i)\in R_0/V=T_0/V\rtimes P$. Now consider the group $X \rtimes P$ appearing in the extension \eqref{equ:XsdPextension}. It has underlying set $V/U \times T_0/V \times P$ and operation given by
\[
(z_1,x_1,p_1)(z_2,x_2,p_2)=(z_1+{}^{p_1}z_2+\tfrac{1}{2}\gamma(x_1,{}^{p_1}x_2), x_1+{}^{p_1}x_2, p_1p_2),
\]
where $(z_i,x_i,p_i)\in V/U \times T_0/V \times P$. Now, choose the section $0\times 1$ for the extension \eqref{equ:XsdPextension} and define the corresponding cocycle by
\[
(0\times 1)(x_1,p_1)+(0\times 1)(x_2,p_2)-(0\times 1)((x_1,p_1)+(x_2,p_2))=(\tfrac{1}{2}\gamma(x_1,{}^{p_1}x_2),0,1),
\]
where $(x_i,p_i)\in T_0/V\rtimes P=R_0/V$. As the Baer sum correspond to adding extension cocycles, the group $G_{\gamma,0}$ has operation
\[
(z_1,x_1,p_1)(z_2,x_2,p_2)=(z_1+{}^{p_1}z_2+\tfrac{1}{2}\gamma(x_1,{}^{p_1}x_2)+\eta(x_1,p_1,x_2,p_2),x_1+{}^{p_1}x_2, p_1p_2).
\]
Fix the bijection $T_0/U\stackrel{\cong}\to V/U\times T_0/V$ given by $y\mapsto (y-s(\pi(y)),\pi(y))$.
Then the operation $+_{\lambda}$ in $T_0/U$ (Definition \ref{defi:Alambda}) is readily checked to be described as
\[
(z_1,x_1)+_\lambda(z_2,x_2)=(z_1+z_2+\tfrac{1}{2}\gamma(x_1,x_2)+\theta(x_1,x_2), x_1+x_2), 
\]
where $(z_i,x_i)\in  V/U\times T_0/V$. Moreover, the action of $p\in P$ on $(z,x)\in V/U\times T_0/V$ is given by 
\[
{}^{p}(z,x)=({}^{p}z+{}^{p}s(x)-s({}^{p}x),{}^{p}x).
\]
It follows that the product in $(T_0/U, +_{\lambda}) \rtimes P$ of $(z_i,x_i,p_i)\in V/U \times T_0/V \times P$ for $i=1,2$ is exactly
\[
(z_1+{}^{p_1}z_2+\tfrac{1}{2}\gamma(x_1,{}^{p_1}x_2)+\theta(x_1,{}^{p_1}x_2)+s({}^{p_1}x_2)-{}^{p_1}s(x_2), x_1+{}^{p_1}x_2, p_1p_2).
\]
Now the lemma follows from Equation \eqref{equ:extensionclassesforinvariantsection}.

\end{proof}


\begin{Rmk}
The above proof shows that, in fact,  the extensions 
\[
V/U \to (T_0/U, +_{\lambda}) \rtimes P \to T_0/V\rtimes P \; \text{and} \; V/U \to G_{\gamma,0} \to T_0/V\rtimes P
\]
are isomorphic.
\end{Rmk}



\subsection{Powerful $p$-central groups and the $\Omega$-extension property}
\label{subsection:pcentralpowerfulomegaextprop}
In this subsection, we state the basic definitions and properties of powerful $p$-central groups as discussed in \cite{Weigel}, as well as we apply them to the context of constructible groups.
\begin{Defi}\label{defi:pcentral}
Let $G$ be a finite group, let $p$ be an odd prime and let $L$ be a $\Z_p$-Lie algebra. Then we say that $G$ or $L$ is \emph{$p$-central} if its elements of order $p$ are contained in its centre:
$$
\Omega_1(G)=\langle g \in G | \; g^p=1\rangle  \subseteq Z(G) \;\; \text{and}\;\; \Omega_1(L)=\langle a\in L | \; pa=0\rangle\subseteq Z(L).
$$
\end{Defi}
\begin{Defi}\label{defi:powerful}
Let $p$ be an odd prime, let $G$ be a finite $p$-group and let $L$ be a $\Z_p$-Lie algebra. We say that $G$ is \emph{powerful} if $[G,G] \subseteq G^p$ and we say that $L$ is \emph{powerful} if $[L,L]\subseteq pL$. 
\end{Defi}
\begin{Defi}\label{defi:omegaEP}
Let $p$ be an odd prime, let $G$ be a $p$-central group and let $L$ be a $p$-central $\Z_p$-Lie algebra. We say that $G$ has the \emph{$\Omega$-extension property} ($\Omega$EP for short) if there exists a $p$-central group $H$ such that $G = H/\Omega_1(H)$. We say that $L$ has the $\Omega$EP if there exists a $p$-central Lie algebra $A$ such that $L=A/\Omega_1(A)$.
\end{Defi}

The following result of Weigel \cite[Theorem 2.1, Corollary 4.2]{Weigel} determines the cohomology ring of powerful, $p$-central $p$-groups with $\Omega$EP. It uses the following subquotient of the cohomology ring:

\begin{Defi}\label{defi:reducedcohomology}
For a finite group $G$, its \emph{reduced} $\bmod\; p$ cohomology ring, that we denote by $H^*(G;\F_p)_{\text{red}}$, is the quotient $H^*(G;\F_p)/nil(H^*(G;\F_p))$, where $nil(H^*(G;\F_p))$ is the ideal of all nilpotent elements in $H^*(G;\F_p)$.
\end{Defi}

%

\begin{Thm} \label{ThmWeigel}Let $p$ be an odd prime and let $G$ be a powerful $p$-central $p$-group with the $\Omega$EP and $d=\dim_{\F_p}(\Omega_1(G))$. Then:
\begin{enumerate}
\item $H^*(G;\F_p) \cong \Lambda[x_1, \dots, x_d] \otimes \F_p[y_1, \dots, y_d]
$ where $|x_i|=1$ and $|y_i|=2$.\label{ThmWeigelHG}
\item The reduced restriction mapping $j_{\red}\colon H^*(G;\F_p)_{\red} \to H^*(\Omega_1(G);\F_p)_{\red}$ is an isomorphism.\label{ThmWeigeljred}
\end{enumerate} 
\end{Thm}

\subsection{Properties of twisted abelian $p$-groups}\label{subsection:furtherpropertiesoftwistedabelian}

We prove additional properties of the group $A_{\lambda}$ introduced in Definition \ref{defi:Alambda}.

\begin{Lem}\label{additionalpropertiesAlambda} With the notations of Definition \ref{defi:Alambda}, the following properties hold:
\begin{enumerate}
\item \label{propertiesAlambdapowerfulcondition}
$A_\lambda$ is powerful if and only if $\im(\lambda)\leq pA$.
\item \label{propertiesAlambdapcentralcondition}
$A_\lambda$ is $p$-central if and only if $\Omega_1(A)\leq \rad(\lambda)$.
\end{enumerate}
\end{Lem}

\begin{proof}
To prove \eqref{propertiesAlambdapowerfulcondition}, note that the commutator of $a,b\in A_\lambda$ is given by $[a,b]=\lambda(a,b)$. Then it follows from Lemma \ref{propertiesAlambda}\eqref{propertiesAlambdasamepowers}. For the last item \eqref{propertiesAlambdapcentralcondition}, note that, by \eqref{propertiesAlambdasameOmega1} in Lemma \ref{propertiesAlambda}, $\Omega_1(A)=\Omega_1(A_\lambda)$, and that $Z(A_\lambda)=\rad(\lambda)$ by the commutator description above. Then, the claim holds.
\end{proof}

Now, recall the notation of Subsection \ref{subsection:uniserialpadicspacegroupsandconstructible} and consider the twisted abelian $p$-group $(T_0/U, +_\lambda)$ already regarded in Lemma \ref{lemma:GalphaGalpha0}. Under mild assumptions this group has nice properties.

\begin{Lem} \label{propertiesT0U+*}
Assume that 
\begin{equation}\label{equation:extraassumptionsUV}
V\leq pT_0
\end{equation}
and suppose that there is a $P$-invariant sublattice $W$ of $T$ such that
$U\leq pW$ and $W\leq pV$. Define $\lambda'\in \Hom_P(\Lambda^2(T_0/W),T_0/W)$ as $\gamma$ precomposed it with $T_0/W\twoheadrightarrow T_0/V$ and postcomposed it with $V/U\to T_0/W$. Then
\begin{enumerate}
\item $(T_0/U, +_\lambda)$ is a powerful $p$-central group, and
\item $(T_0/W,+_{\lambda'})$ is a powerful $p$-central group with $\Omega$EP.
\label{propertiesT0U+*TOWOmegaEP}
\end{enumerate} 
\end{Lem}

\begin{proof} 
We have $\im(\lambda)\leq V/U\leq \rad(\lambda)$ and $\im(\lambda')\leq V/W\leq \rad(\lambda')$. Moreover, $V/U\leq (pT_0)/U=p(T_0/U)$ and $V/W\leq (pT_0)/W=p(T_0/W)$. Finally, $\Omega_1(T_0/U)=(\frac{1}{p}U)/U\leq (\frac{1}{p^2}U)/U\leq V/U$ and $\Omega_1(T_0/W)=(\frac{1}{p}W)/W\leq V/W$. Then it follows from Lemma \ref{additionalpropertiesAlambda}\eqref{propertiesAlambdapowerfulcondition} and Lemma \ref{additionalpropertiesAlambda}\eqref{propertiesAlambdapcentralcondition} that both  $(T_0/U,+_\lambda)$ and $(T_0/W,+_{\lambda'})$ are powerful $p$-central groups. The same arguments show that $(T_0/W',+_{\lambda''})$ is powerful $p$-central for $W'=pW$ and $\lambda''$ defined analogously. 
Now, by Lemma \ref{propertiesAlambda}\eqref{propertiesAlambdasameOmega1}, $\Omega_1((T_0/W')_{\lambda''})=\Omega_1(T_0/W')$ and this group is exactly $W/W'$. It is straightforward that 
\[
(T_0/W,+_{\lambda'})\cong (T_0/W',+_{\lambda''})/(W/W')
\]
and hence $(T_0/W,+_{\lambda'})$ has the $\Omega$EP. 
\end{proof}

\subsection{Standard uniserial $p$-adic space group}\label{subsection:standarduniserialpadicspacegroups}
Recall that in Subsection \ref{subsection:uniserialpadicspacegroupsandconstructible} we considered a uniserial $p$-adic space group, $1 \to T \to R \to P \to 1$, and we set  $\tilde{T}:= T \otimes_{\Z_p} \Q_p$. Since the action of $P$ on $T$ is faithful, $P$ can be embedded in $\GL(\tilde{T})=\GL_{d_x}(\Q_p)$, and hence $P$ is contained in a maximal $p$-subgroup of  $\GL_{d_x}(\Q_p)$. For $p$ odd, $\GL_{d_x}(\Q_p)$ has a unique maximal $p$-subgroup, $W(x)$, up to conjugation which is an iterated wreath product:
\[
W(x)=C_p\wr \overbrace{C_p \wr \cdots \wr C_p}^{x-1}.
\]
The action of $W(x)$ on the $p$-adic lattice $\Z_p^{d_x}$ is described as follows: the leftmost copy of $C_p$ is generated by  the companion matrix of the polynomial $y^{p-1}+ \cdots + y+1$, i.e,  by the following matrix:
\begin{equation}
\label{eq:action_maximal_class}
M=\begin{pmatrix}
0 &0 & \cdots &0 &-1 \\
1 &0 & \cdots &0 &-1 \\
0 &1 & \cdots &0 &-1 \\
\vdots & \vdots & &\vdots &\vdots \\
0 &0 &\cdots &1 &-1
\end{pmatrix} \in \GL_{p-1}(\Z).
\end{equation}
The remaining $(x-1)$ copies of $C_p$ act by permutation matrices as $\overbrace{C_p \wr \cdots \wr C_p}^{x-1}$ is the Sylow $p$-subgroup of the symmetric group $\Sigma_{p^{x-1}}$.

\begin{Rmk}\label{remark:sylowpeven}
For $p=2$, we have $d_x=2^{x-1}$ and there is another conjugacy class $\tilde{W}(x)$ in $\GL_{2^{x-1}}(\Q_2)$ given as follows:
\[
\tilde{W}(x)=Q_{16}\wr \overbrace{C_2 \wr \cdots \wr C_2}^{x-3},
\]
where $Q_{16}$ denotes the quaternion group of order 16. The action of $\tilde{W}(x)$ on $\Z_2^{d_x}$ is described in \cite{LeedGreenPleskenV}.  
\end{Rmk}

 For odd $p$ and after a suitable conjugation, we may assume that $T_0$ is an $W(x)$-invariant lattice with $T_0\leq \Z_p^{d_x}$. Hence, for $p$ odd we have 
\begin{equation}\label{equ:R_0containedZW(x)}
R_0=T_0\rtimes P\leq \Z_p^{d_x}\rtimes W(x),
\end{equation}
where both $T_0\rtimes P$ and $\Z_p^{d_x}\rtimes W(x)$  are split uniserial $p$-adic space groups. We call $\Z_p^{d_x}\rtimes W(x)$ the \emph{standard uniserial $p$-adic space group} of dimension $d_x$.

\section{Counting theorems}
\label{section:countingtheorems}

Throughout this section we use the sectional rank of a $p$-group $G$ (see \cite[\S11]{khukhro}):
\begin{align*}
\rk(G):&=\text{max}\{d(H) | H \leq G\}\\
&=\text{max}\{ \text{dim}_{\F_p}(H/N) | \; N \unlhd H \leq G, \; \text{where} \; H/N \; \text{is elementary abelian}\},
\end{align*}
where $d(H)$ denotes the number of minimal generators of $H$. The following properties are standard and can be found in \cite[\S 4 and \S11]{khukhro}. We shall use them without further notice.
\begin{enumerate}
\item If $G$ is a powerful $p$-group, then $\rk(G)=d(G)$.
\item If $N$ is a subgroup of $G$ then $\rk(N)\leq \rk(G)$.
\item If $N$ is a normal subgroup of $G$ then $\rk(G)\leq \rk(G/N) + \rk(N)$.
\end{enumerate}

Next, we recall some results of J.F. Carlson from \cite{Carlson} and we prove some natural generalizations. 

\begin{Thm}[{\cite[Theorem 2.1]{Carlson}}]\label{Car2}
Let $R$ be a finitely generated, graded-commutative $\F_p$-algebra and
let $S$ be the bigraded algebra induced by some filtration of $R$. Then the algebra structure of $R$ is determined by the algebra structure of $S$ within a finite number of possibilities. 
\end{Thm}

\begin{Thm}[{\cite[Theorem 3.5]{Carlson}}]\label{Car1} Let $n$ be a positive integer. Suppose that $S$ is a finitely generated $k$-algebra. Then, there are only finitely many $k$-algebras $R$ with the property that $R \cong H^*(H,k)$ for $H$ a subgroup of a $p$-group $G$ with $H^*(G;k)\cong S$ and $|G:H| \leq p^n.$
\end{Thm}


Now we need to generalize \cite[Theorem 3.3]{Carlson}. We start with the following two lemmas.

\begin{Lem}
\label{lema:liftingLie}
Let $L$ be a finite $\Z_p$-Lie algebra of nilpotency class two. Then there exist $\tilde{L}$ a finite powerful $p$-central nilpotent $\Z_p$-Lie algebra of nilpotency class $2$ such that $pL\cong \tilde{L}/\Omega_1(\tilde{L})$.
\end{Lem}

\begin{proof}
Consider the $\Z_p$-Lie algebra $L$ as a $d$-generated abelian $p$-group. Then, $L$ is isomorphic to a quotient of the free $\Z_p$-module $M$ of rank $d$ by a submodule $I$ which is contained in $pM$. The Lie bracket in $L$ is an antisymmetric bilinear form $[,]\colon L\times L\to L$ and therefore, it can be lifted to an antisymmetric bilinear form $\{ ,\}\colon M\times M\to M$. A direct computation shows that $\{ ,\}$ satisfies the Jacobi identity in $\tilde{L}=pM/pI$. Furthermore, $(\tilde{L},+, \{,\})$ is a powerful $p$-central $\Z_p$-Lie algebra. 
By the third isomorphism theorem 
$$\frac{\tilde{L}}{\Omega_1(\tilde{L})}=\frac{pM/pI}{I/pI}\cong \frac{pM}{I}=pL,$$
and this concludes the proof.
\end{proof}

\begin{Lem}\label{lemma:CpGQ}
Let $p$ be an odd prime and let $G$ be a $p$-group with $\rk(G)\leq r$, and let
\begin{equation}
 \xymatrix{
1 \ar[r] &C_p \ar[r] &G\ar[r]^{\pi} &Q \ar[r] &1,}
\end{equation}
be an extension of groups. Suppose that $Q$ has a subgroup $A$ of nilpotency class $2$. Set $B=\pi^{-1} (A)^{p^2}$. Then $B$ is a powerful $p$-central $p$-group of nilpotency class $2$ with the $\Omega$EP and $|G:B|\leq p^{4r}|Q:A|$.
\end{Lem}
\begin{proof}
Put $C=\pi^{-1} (A)$, $D=C^p$ and let $N$ be the image of $C_p$ in $G$,
then $[D,D,D]=[C^p,C^p,C^p]=[C,C,C]^{p^3}\subseteq N^{p^3}=1$ (see \cite[Theorem 2.4]{GustavoJonAndrei} and \cite[Theorem 2.10]{Gustavo2000}). In particular $D$ is a $p$-group of nilpotency class $2$. By the Lazard correspondence (for general theory see \cite[\S9 and \S10]{khukhro} and for the explicit formulae see \cite{CicaloGraafLee}), $D=\textbf{exp} (L)$ where $L$ is a $\Z_p$-Lie algebra of nilpotency class two. By Lemma \ref{lema:liftingLie} and again by the Lazard correspondence $B=D^p=\textbf{exp} (pL)$ is a powerful $p$-central $p$-group of nilpotency class $2$ with the $\Omega$EP. Indeed, $pL\cong \tilde{L}/\Omega_1 (\tilde{L})$ where $\tilde{L}$ is a powerful $p$-central Lie algebra of nilpotency class $2$ and by applying the functor $\textbf{exp}$ we obtain that $\textbf{exp} (pL)=\textbf{exp} (\tilde{L})/\textbf{exp}(\Omega_1 (\tilde{L}))$.

We also have $|G:B|=|G:C||C:D||D:B|$ and $|G:C|=|Q:A|$, where $C/D$ and $D/B$ have exponent $p$. Moreover, $C$ has rank at most $r$ and nilpotency class at most $3$. We may write 
\[
|C:D|=|C: \Phi(C)||\Phi(C):D| \; \text{and} \; |D:B|=|D:\Phi(D)||\Phi(D):B|.
\]
Notice that by Burnside Base Theorem in \cite[Theorem 4.8]{khukhro}, the Frattini factor group $C/\Phi(C)$ is an elementary abelian $p$-group and thus, $|C: \Phi(C)| \leq p^{r}$. Also, it is not hard to check that $\Phi(C)/D$ is an elementary abelian group. Then, $|\Phi(C):D|\leq p^r$. Hence, $|C:D| \leq p^{2r}$ and similarly, we obtain that $|D:B|\leq p^{2r}$. Then, the bound in the statement follows.
\end{proof}

\begin{Thm}\label{thm:cohringsfromquotient}
Let $p$ be an odd prime and suppose that 
\begin{equation}
\label{eq:cohringsfromquotient}
1 \to H \to G \to Q \to 1
\end{equation}
is an extension of finite $p$-groups with $|H|\leq n$, $\rk(G)\leq r$ and $Q$ has a subgroup $A$ of nilpotency class $2$ with $|Q:A| \leq f$. Then the ring $H^*(G;\F_p)$ is determined up to a finite number of possibilities (depending on $p$, $n$, $r$ and $f$) by the ring $H^*(Q;\F_p)$.
\end{Thm}

\begin{proof}
If $H\cong C_p$ then, by Lemma \ref{lemma:CpGQ}, there exists a powerful $p$-central subgroup $B$ of $G$ with $\Omega$EP and whose index is bounded in terms of $p$, $r$ and $f$. If $H$ is not contained in $B$ we consider $H\times B$ instead of $B$. In both situations, there exists an element $\eta\in H^2(B,\F_p)$ (resp. $\eta\in H^2(B\times H,\F_p)$) such that $\text{res}^B_H( \eta )$ is non-zero (resp. $\text{res}^{B\times H}_H( \eta )$) (see Theorem \ref{ThmWeigel}\eqref{ThmWeigeljred}). Then the spectral sequence arising from 
\[
1 \to H \to G \to Q \to 1
\]
stops at most at the page $2|G:B|+1$ (cf. \cite[proof of Lemma 3.2]{Carlson}). Now the theorem holds by \cite[Proposition 3.1]{Carlson}.

For general $H$, we proceed by induction on $|H|$. Suppose that the result holds for all the group extensions of the form
\[
1\to H'\to G\to Q,
\] 
where $|H'|<|H|\leq n$, $\rk(G)\leq r$ and with $A\leq Q$ of nilpotency class $2$ and $|Q:A|\leq f$. Choose a subgroup $H'\leq H$ with $H' \unlhd G$ and $|H:H'|=p$. The quotients $G'=G/H'$ and $C_p\cong H/H'$ fit in a short exact sequence
\begin{equation}\label{eq:applyinginductionextensiongroups}
1 \to C_p \to G' \overset{\pi}\to Q \to 1
\end{equation}
and we also have the following extension of groups,
\begin{equation}\label{eq:centralextensioninduction}
1\to H'\to G\to G'\to 1.
\end{equation}
Applying Lemma \ref{lemma:CpGQ} to the extension of groups \eqref{eq:applyinginductionextensiongroups}, we know that $G'$ has a $p$-subgroup $B'$ of nilpotency class two. Moreover, as $\rk(G')\leq \rk(G)$, we have that $|G':B'|\leq p^{4r}|Q:A|$. Also, by the previous case, we have that the cohomology algebra $H^*(G';\F_p)$ is determined up to a finite number of possibilities (depending on $p$, $n$, $r$ and $f$) by the algebra $H^*(Q;\F_p)$.

Now, we may apply the induction hypothesis to the extension \eqref{eq:centralextensioninduction} since $|H'|<|H|$. Then, the cohomology algebra $H^*(G;\F_p)$ is determined up to a finite number of possibilities (depending on $p$, $n$, $r$ and $f$) by the algebra $H^*(G';\F_p)$. In turn, the result holds.

\end{proof}


\section{Cohomology of $p$-groups that split over an abelian subgroup}
\label{section:cohomologyabelian}
Assume that $G$ is a $p$-group that splits over an abelian subgroup $K$ of rank $d$, $K = C_{p^{i_1}} \times \dots \times C_{p^{i_d}}$. In this section we show that if $d<p$ and the action $G/K\to \Aut(K)$ has an integral lifting (defined below), then there are finitely many possibilities for the cohomology ring $H^*(G;\F_p)$ for all infinitely many choices of the exponents $i_l$. Here we are considering cohomology with trivial coefficients in the field of $p$ elements, $\F_p$, and we assume that either $p$ is odd or $p=2$ and $i_l>1$ for all $l$.

\subsection{Cohomology of abelian $p$-groups of small rank}\label{subsection:cohoabeliansmallrank}
Consider the abelian groups $K = C_{p^{i_1}} \times \dots \times C_{p^{i_d}}$ and $K'= C_{p^{j_1}} \times \dots \times C_{p^{j_d}}$ for some fixed rank $0<d$. Then there is an abstract isomorphism of cohomology rings
\begin{equation}\label{eq:abstractisomorphismabeliangroupssamerank}
H^*(K;\F_p)\cong H^*(K';\F_p).
\end{equation}
When $d<p$, we shall realize this isomorphism via quasi-isomorphisms in the category of cochain complexes. To that end, note that both cohomology rings above are isomorphic to the tensor product of a polynomial algebra with an exterior algebra: $\F_p[x_1,\ldots,x_d]\otimes \Lambda(y_1,\ldots,y_d)$, where $\deg(x_l)=2$ and $\deg(y_l)=1$. 

The Pr{\"u}fer $p$-group $C_{p^{\infty}}=\bigcup _{k\geq 1} C_{p^k}$ is an infinite discrete abelian group whose mod-$p$ cohomology ring is polynomial with one generator in degree $2$ (See \cite[V.6.6]{KSBrown82} and recall that $C_{p^{\infty}}$ is divisible). Hence we also have $H^*(C^d_{p^{\infty}};\F_p)= \F_p[x_1,\ldots,x_d]$ with $\deg(x_l)=2$. Moreover, the group inclusion $K \hookrightarrow  C^d_{p^{\infty}}$ induces a cochain map 
\begin{equation}\label{equniversalmapeven}
C^*(C^d_{p^{\infty}};\F_p)\stackrel{\varphi_e}\longrightarrow C^*(K;\F_p)
\end{equation}
that becomes  an isomorphism on reduced cohomology (Definition \ref{defi:reducedcohomology}), i.e, on cohomology modulo the ideal of the nilpotent elements. Note that the map \eqref{equniversalmapeven}  preserve the standard cup products in $C^*(C^d_{p^{\infty}};\F_p)$ and $C^*(K;\F_p)$ (see Subsection \ref{subsection:standardcupproduct}).  

By abusing notation, we denote by $\Lambda(y_1,\ldots,y_d)$ both the exterior algebra (with product denoted by $\cup$) and the cochain complex obtained by equipping it with the $0$ differential. Consider then the following cochain map introduced in \cite[Proof of Proposition 2]{Taelman2014}:
\begin{equation}\label{eq:universalmapodd}
\Lambda(y_1,\ldots,y_d)\stackrel{\varphi_o}\longrightarrow C^*(K;\F_p)
\end{equation}
that on degree $t$, with $0\leq t\leq d$, sends  $y_{l_1}\cdots y_{l_t}$ to the element of $C^t(K;\F_p)$
\[
\frac{1}{t!}\sum_{\sigma\in \Sigma_t} \sgn(\sigma) Y_{l_{\sigma(1)}}\cup \ldots \cup Y_{l_{\sigma(t)}},
\]
where $Y_i$ are the representatives of cohomology classes generating $H^1(K;\F_p)$ defined by 
\begin{equation}\label{eq:generatorsY_i}
Y_i(k_0,k_1)=\overline{(k_1-k_0)_i},
\end{equation}
where $\overline{k_l}$ denotes the image by the reduction $C_{p^{i_l}}\twoheadrightarrow C_p$ of the $l$-th coordinate $k_l$ of $k\in K$. Note that the condition $d<p$ is needed in the definition of this cochain map. When we pass to cohomology we obtain the identity on degree $1$ . Note that $H^*(\varphi_0)$ is a homomorphism of rings because of the equality 
\[
[\frac{1}{t!}\sum_{\sigma\in \Sigma_t} \sgn(\sigma) Y_{l_{\sigma(1)}}\cup \ldots \cup Y_{l_{\sigma(t)}}]=[Y_{l_1}\cup \ldots \cup Y_{l_t}],
\]
which in turn follows from the fact that $[Y_{l_{\sigma(1)}}\cup \ldots \cup Y_{l_{\sigma(t)}}]=\sgn(\sigma)[Y_{l_1}\cup \ldots \cup Y_{l_t}]$ in the graded commutative algebra $H^*(K;\F_p)$. 

\begin{Rmk}
Consider the cochain map defined by fixing an $\F_p$ basis for $\Lambda(y_1,\ldots,y_d)$ and extending linearly the map $y_{l_1}\cdots y_{l_t}\to Y_{l_1}\cup \ldots \cup Y_{l_t}$.
It is simpler and good enough for the purposes of Lemma \ref{lem:chainmapfromuniversal} below. Nevertheless, it is not invariant, a crucial condition that $\varphi_o$ satisfies and that will be needed later on in Lemma \ref{lem:chainmapfromuniversalinvariant}. 
\end{Rmk}

\begin{Defi}
Let $p$ be a prime and let $d<p$. Define $U(p,d)$ as the cochain complex $C^*(C^d_{p^{\infty}};\F_p)\otimes \Lambda(y_1,\ldots,y_d)$.
\end{Defi}

\begin{Lem}\label{lem:chainmapfromuniversal}
For every prime $p$ and any abelian $p$-group $K$ of rank $d<p$ there exists a quasi-isomorphism
\[
U(p,d)\stackrel{\varphi}\longrightarrow C^*(K;\F_p)
\]
that induces an isomorphism of rings $H^*(U(p,d))\cong H^*(K;\F_p)$.
\end{Lem}
\begin{proof}
Define 
\[
U(p,d)=C^*(C^d_{p^{\infty}};\F_p)\otimes \Lambda(y_1,\ldots,y_d)\stackrel{\varphi_e\otimes \varphi_o}\longrightarrow C^*(K;\F_p)\otimes C^*(K;\F_p)\stackrel{\cup}\to C^*(K;\F_p),
\]
where the first arrow is the tensor product of the cochain maps \eqref{equniversalmapeven} and \eqref{eq:universalmapodd} and the second map is the standard cup product (Subsection \ref{subsection:standardcupproduct}). Then the claim follows by the properties of the cochain maps \eqref{equniversalmapeven} and \eqref{eq:universalmapodd}.
\end{proof}

\begin{Cor}\label{cor:abeliansamecohomology}
For every prime $p$ and any two abelian $p$-groups $K$ and $K'$ of rank $d<p$ there exists a zig-zag of quasi-isomorphisms
\[
C^*(K;\F_p)\leftarrow U(p,d)\to C^*(K';\F_p)
\]
that induce isomorphisms of rings $H^*(K;\F_p)\cong H^*(U(p,d))\cong H^*(K';\F_p)$.
\end{Cor}

Note that this zig-zag realizes the ring isomorphism \eqref{eq:abstractisomorphismabeliangroupssamerank}.
\subsection{Cohomology of $p$-groups that split over an abelian $p$-group of small rank}
\label{subsection:cohomologysmallrank}
Suppose  that $P$ is a $p$-group acting on the abelian $p$-group $K=C_{p^{i_1}} \times \dots \times C_{p^{i_d}}$ via $P\stackrel{\alpha}\to \Aut(K)$.  Set $R$ to be the subring of $\End(\Z^d)=\M_d(\Z)$ consisting of integral matrices $A=(a_{n,m})$ such that each entry $a_{n,m}$ is divisible by $\max(p^{i_n-i_m},1)$. Then there is a surjective ring homomorphism $w\colon R\to \End(K)$  \cite{HillarRhea2007}. Moreover, if $A\in R$, then
\begin{equation}\label{equ:walpha}
w(A)(\pi(n_1,\ldots,n_d))=\pi(\sum_l a_{1,l}n_l,\ldots,\sum_l a_{d,l}n_l)=(\sum_l a_{1,l}\overline{n_l},\ldots,\sum_l a_{d,l}\overline{n_l}),
\end{equation}
where $\pi$ is the surjection $\Z^d\twoheadrightarrow K$ that takes $(n_1,\ldots,n_r)$ to the element $\pi(n_1,\ldots,n_d)=(\overline{n_1},\ldots,\overline{n_d})$ obtained by reducing modulo $p^{i_l}$ the $l$-th coordinate. For instance, if $i=i_1=\ldots=i_d$, then $w$ takes the matrix $A$ to its $C_{p^i}$ reduction in $\M_d(C_{p^i})=\End(K)$.

\begin{Defi}\label{defi:integrallifting}
We say that the homomorphism $P\stackrel{\tilde \alpha}\to R$ is an \emph{integral lifting} of the action $P\stackrel{\alpha}\to \Aut(K)$ if the following diagram commutes,
\[
\xymatrix{
P\ar[d]_{\alpha}\ar[r]^{\tilde\alpha}& R\ar[d]^w\\
\Aut(K)\ar@{^(->}[r]& \End(K),
}
\]
where the bottom horizontal arrow is the inclusion. Note that, as $P$ is a group, the image of $\tilde\alpha$ lies in the group of units of $R$. 
\end{Defi}

\begin{Rmk}\label{rmk:PactsonB_*andC^*}
The group $P$ acts on the chain complex $B_*(K;\F_p)$ via $p \cdot (g_0, \dots, g_n)=(p\cdot g_0, \dots,p\cdot g_n)$, where $p\in P$ and $(g_0, \dots, g_n)\in B_n(K;\F_p)$. There is also an action  of $P$ on $C^*(K;\F_p)$ given by  $(p\cdot f)(g)=f(p^{-1}\cdot g)$, where $p\in P$, $g \in B_n(K;\F_p)$ and $f \in C^n(K;\F_p)$. 
\end{Rmk}

The next result is a $P$-invariant version of Lemma \ref{lem:chainmapfromuniversal}.

\begin{Lem}\label{lem:chainmapfromuniversalinvariant}
Let $p$ be a prime, let $K$ be an abelian $p$-group of rank $d<p$ and let $P$ be a $p$-group that acts on $K$ via $P\to \Aut(K)$. If the action has an integral lifting then $P$ acts on $U(p,d)$ and the quasi-isomorphism
\[
U(p,d)\to C^*(K;\F_p)
\]
is $P$-invariant.
\end{Lem}
\begin{proof}
Recall that $U(p,d)=C^*(C^d_{p^{\infty}};\F_p)\otimes \Lambda(y_1,\ldots,y_d)$ and consider the integral lifting $\tilde\alpha\colon P\to R\subseteq \M_d(\Z)$. We shall show that $P$ acts on $C^*(C^d_{p^{\infty}};\F_p)$ and on $\Lambda(y_1,\ldots,y_d)$ and hence diagonally on $U(p,d)$. Choose $p\in P$ and set $\tilde\alpha(p)=A=(a_{n,m})$. Then $A$ acts on $C^d_{p^{\infty}}$ via
\[
p\cdot (m_1,\ldots,m_d)= (\sum_l a_{1,l}m_l,\ldots,\sum_l a_{d,l}m_l),
\]
on $B_n(C^d_{p^{\infty}};\F_p)$ via $p\cdot (g_0,\ldots,g_n)= (p\cdot g_0,\ldots, p\cdot g_n)$ and on $f\in C^n(C^d_{p^{\infty}};\F_p)$ via $(p\cdot f)(g)=f(p^{-1}\cdot g)$. With these definitions, the cochain map \eqref{equniversalmapeven} $C^*(C^d_{p^{\infty}};\F_p)\to C^*(K;\F_p)$ is $P$-invariant because, by \eqref{equ:walpha}, the inclusion $K\hookrightarrow C^d_{p^{\infty}}$ is $P$-invariant. 

For the action of $p$ on $c_1y_1+\ldots+c_dy_d\in \Lambda^1(y_1,\ldots,y_d)$, where $c_l\in \F_p$, write $\tilde\alpha(p^{-1})=B=(b_{n,m})$ and set
\[
p\cdot (c_1y_1+\ldots+c_dy_d)= (\sum_l b_{l,1}c_l)y_1+ \ldots+(\sum_l b_{l,d}c_l)y_d.
\]
Observe that this action extends to $\Lambda(y_1,\ldots,y_d)$. Moreover, on $C^*(K;\F_p)$ we also have that
\begin{equation}\label{equ:PactsonY_i}
p\cdot (c_1Y_1+\ldots+c_dY_d)= (\sum_l b_{l,1}c_l)Y_1+ \ldots+(\sum_l b_{l,d}c_l)Y_d.
\end{equation}
Then it turns out after a routine computation that the cochain map \eqref{eq:universalmapodd} is $P$-invariant.
\end{proof}

Next we present the main result of this section.

\begin{Prop}\label{prop:commonintegralliftingsamecohomology}
Let $p$ be a prime and let $\{G_i=K_i\rtimes P\}_{i\in I}$ be a family of groups such that $K_i$ is abelian of fixed rank $d<p$ for all $i$ and that all actions of $P$ have a common integral lifting. Then the following holds:
\begin{enumerate}
\item The graded $\F_p$-modules $H^*(G_i;\F_p)$ and $H^*(G_{i'};\F_p)$ are isomorphic for all $i,i'$.
\item There is a filtration of $H^*(G_i;\F_p)$ for each $i$ and the associated bigraded algebras are isomorphic for all $i$.
\item There are finitely many isomorphism types of rings in the collection of cohomology rings $\{H^*(G_i;\F_p)\}_{i\in I}$.
\end{enumerate}
\end{Prop}
\begin{proof}
Choose $i_0\in I$ and set $G=G_{i_0},K=K_{i_0}$. Then take any $i\in I$ and let $G'=G_i,K'=K_i$. By Subsection \ref{subsection:resolution-semidirect-product}, we have $H^*(G;\F_p)\cong H^*(\tot(D_C))$,  where $D_C$ is the double complex \[
D_C^{*,*}=\Hom_{\F_pP}(B_*(P;\F_p), C^*(K;\F_p)).
\]
Similarly, we have $H^*(G';\F_p)=H^*(\tot(D_{C'}))$ for $D_{C'}$ the analogous double complex. Now, by Lemma \ref{lem:chainmapfromuniversalinvariant}, there exists a zig-zag of $P$-invariant quasi-isomorhpisms 
\[
C^*(K;\F_p)\stackrel{\varphi}\leftarrow U(p,d)\stackrel{\varphi'}\to C^*(K';\F_p).
\]
Then two applications of Lemma \ref{lemma:homalg} gives immediately that the maps
\[
\tot(D_C)\stackrel{\tot(\varphi_*)}\longleftarrow \tot(D_U)\stackrel{\tot(\varphi'_*)}\longrightarrow \tot(D_{C'})
\]
are both quasi-isomorphisms, where $D^{*,*}_U=\Hom_{\F_pP}(B_*(P;\F_p),U^*(p,d))$. In particular, $H^*(G;\F_p)\cong H^*(G';\F_p)$ as graded $\F_p$-vector spaces. Moreover, we can endow $\tot(D_C)$, $\tot(D_{C'})$ and $\tot(D_U)$ with a product by Equation \eqref{equ:productforHom}. Then, by Lemma \ref{lemma:homalgalgebra}, there are filtrations of $H^*(\tot(D_C))$, $H^*(\tot(D_U))$ and $H^*(\tot(D_{C'}))$ with isomorphic associated bigraded algebras. Indeed, by Subsection \ref{subsection:resolution-semidirect-product}, the products in $H^*(\tot(D_C))\cong H^*(G;\F_p)$ and 
$H^*(\tot(D_{C'}))\cong H^*(G';\F_p)$ are the usual cup products for the cohomology rings of $G$ and $G'$. Finally, the product in the corresponding bigraded algebras is induced by the products in $H^*(G;\F_p)$ and $H^*(G';\F_p)$. Now, by Theorem \ref{Car2}, there are finitely many possibilities for the rings $H^*(G;\F_p)$ and $H^*(G';\F_p)$.
\end{proof}

\subsection{Cohomology of $p$-groups of unbounded rank.}
\label{subsection:cohomologyunboundedrank}
In this subsection, we generalize the last two subsections to a situation where the abelian subgroup has unbounded rank. So let  $K = C_{p^{i_1}} \times \dots \times C_{p^{i_d}}$ with $d\leq p-1$ and consider the $n$-fold direct product  $L=K\times \ldots\times K$. Considering $L$ as an iterated semi-direct product with trivial action, we obtain from Subsection \ref{subsection:resolution-semidirect-product} that the cohomology of the following cochain complex,
\[
C^*_\times(L;\F_p)=\Hom_{\F_pL}(B_*(K;\F_p)\otimes\ldots\otimes B_*(K;\F_p),\F_p),
\] 
is exactly $H^*(L;\F_p)$. More explicitly, the action of $L$ is given by 
\[
(k_1,\ldots,k_n)\cdot (z_1\otimes \ldots \otimes z_n)=k_1\cdot z_1\otimes \ldots \otimes k_n\cdot z_n,
\]
where $k_i\in K$, $z_i\in B_*(K;\F_p)$ and the action of $K$ on $B_*(K;\F_p)$ was defined in Subsection \ref{subsection:standardresolution}. Moreover, there is a product on $C^*_\times(L;\F_p)$ which induces the usual cup product in $H^*(L;\F_p)$. Here it is the generalization of Lemma \ref{lem:chainmapfromuniversal}. 

\begin{Lem}\label{lem:chainmapfromuniversalunbounded}
There exists a cochain complex $U(p,d,n)$ and a cochain map 
\[
U(p,d,n)\stackrel{\phi}\longrightarrow C^*_\times(L;\F_p)
\]
that induces an isomorphism of rings $H^*(U(p,d,n))\cong H^*(L;\F_p)$.
\end{Lem}
\begin{proof}
Consider the analogous construction for $L_\infty=\stackrel{n}\times C^d_{p^{\infty}}$, i.e., 
\[
C^*_\times(L_\infty;\F_p)=\Hom_{\F_pL_\infty}(B_*(C^d_{p^{\infty}};\F_p)\otimes\ldots\otimes B_*(C^d_{p^{\infty}};\F_p),\F_p).
\]
Then there is a cochain map $\phi_e\colon C^*_\times(L_\infty;\F_p)\to C^*_\times(L;\F_p)$ induced by the $n$-fold tensor product of the group inclusion $K\hookrightarrow C^d_{p^\infty}$. This map  becomes  an isomorphism on reduced cohomology. Next, there are representatives of the generators of $H^1(L;\F_p)$ of the form $1\otimes\ldots\otimes Y_i\otimes 1\ldots\otimes 1$, where $Y_i$ was defined in Equation \eqref{eq:generatorsY_i}. Now consider the cochain map $\phi_o\colon \stackrel{n}\otimes \Lambda(y_1,\ldots,y_d)\to C^*_\times(L;\F_p)$ given by $\phi_o=\stackrel{n}\otimes \varphi_o$ and where $\varphi_o$ was defined in Equation \eqref{eq:universalmapodd}. Finally, the cochain map
\[
U(p,d,n)=C^*_\times(L_\infty;\F_p)\bigotimes \stackrel{n}\otimes \Lambda(y_1,\ldots,y_d)\stackrel{\phi_e\otimes \phi_o}\longrightarrow C^*_\times(L;\F_p)\otimes C^*_\times(L;\F_p)\stackrel{\cup}\to C^*_\times(L;\F_p),
\]
induces an isomorphism of cohomology rings.
\end{proof}

Now we consider a semidirect product $L\rtimes Q$, where $Q$ is a subgroup of $P\wr S$. Here $P\stackrel{\alpha}\to \Aut(K)$ acts on $K$, $S$ is a subgroup of the symmetric group $\Sigma_n$ and the action of $Q$ on $L$ is given by
\[
(p_1,\ldots,p_n,\sigma)\cdot (k_1,\ldots,k_n)=(p_1\cdot k_{\sigma^{-1}(1)},\ldots,p_n\cdot k_{\sigma^{-1}(n)}).
\]
Next we prove a generalization of Proposition \ref{prop:commonintegralliftingsamecohomology}.

\begin{Prop}\label{prop:commonintegralliftingsamecohomologyunbounded}
Let $p$ be a prime and let $\{G_i=L_i\rtimes Q\}_{i\in I}$ be a family of groups such that $L_i=K_i\times \ldots\times K_i$, $K_i$ is abelian of fixed rank $d<p$, $Q\leq P\wr S$ and all  actions of $P$ have a common integral lifting. Then the following holds:
\begin{enumerate}
\item The graded $\F_p$-modules $H^*(G_i;\F_p)$ and $H^*(G_{i'};\F_p)$ are isomorphic for all $i,i'$.
\item There is a filtration of $H^*(G_i;\F_p)$ for each $i$ and the associated bigraded algebras are isomorphic for all $i$.\label{prop:commonintegralliftingsamecohomologyunboundedbigradedalgebras}
\item There are finitely many isomorphism types of rings in the collection of cohomology rings $\{H^*(G_i;\F_p)\}_{i\in I}$.\label{prop:commonintegralliftingsamecohomologyunboundedcohorings}
\end{enumerate}
\end{Prop}
\begin{proof}
The proof is identical to that of Proposition \ref{prop:commonintegralliftingsamecohomology} once we prove  that the cochain map $\phi\colon U(p,d,n)\longrightarrow C^*_\times(L_i;\F_p)$ from Lemma \ref{lem:chainmapfromuniversalunbounded} is $Q$-invariant for some action of $Q$ on $U(p,d,n)$. First, $Q$ acts on $B_*(K;\F_p)\otimes \ldots \otimes B_*(K;\F_p)$ by 
\[
(p_1,\ldots,p_n,\sigma)\cdot (z_1\otimes \ldots \otimes z_n)=(-1)^\epsilon (p_1\cdot z_{\sigma^{-1}(1)}\otimes \ldots\otimes p_n\cdot z_{\sigma^{-1}(n)}),
\]
where $z_i\in B_*(K;\F_p)$. Signs must be chosen appropriately in order for this action to commute with the differential, and it is enough to choose $\epsilon$ as for wreath products in \cite[p. 49]{LEvens91}. Moreover, this action satisfies points $(1)$ and $(2)$ in Subsection \ref{subsection:resolution-semidirect-product} and $Q$ acts on $C^*_\times(L_i;\F_p)$ by $(q\cdot f)(z)=f(q^{-1}\cdot z)$. In the proof of Proposition \ref{prop:commonintegralliftingsamecohomology}, we showed how $P$ acts, via its integral lifting, on $B_*(C^d_{p^\infty};\F_P)$ and on $\Lambda(y_1,\ldots,y_d)$. Then $Q$ acts $B_*(C^d_{p^{\infty}};\F_p)\otimes\ldots\otimes B_*(C^d_{p^{\infty}};\F_p)$ by
\[
(p_1,\ldots,p_n,\sigma)\cdot (z_1\otimes \ldots \otimes z_n)=(-1)^\epsilon (p_1\cdot z_{\sigma^{-1}(1)}\otimes \ldots\otimes p_n\cdot z_{\sigma^{-1}(n)}),
\]
where $z_i\in B_*(C^d_{p^{\infty}};\F_p)$ and the signs are chosen as above, and on $C^*_\times(L_\infty;\F_p)$ by $(q\cdot f)(z)=f(q^{-1}\cdot z)$. Moreover, $Q$ acts on $\stackrel{n}\otimes \Lambda(y_1,\ldots,y_d)$ by
\[
(p_1,\ldots,p_n,\sigma)\cdot (z_1\otimes \ldots \otimes z_n) =(-1)^\epsilon (p_1\cdot z_{\sigma^{-1}(1)}\otimes \ldots \otimes p_n\cdot z_{\sigma^{-1}(n)}),
\]
where $z_i\in \Lambda(y_1,\ldots,y_d)$ and the signs are chosen as above. Finally, $Q$ acts diagonally on $U(p,d,n)$ and it is straightforward that $\phi$ is $Q$-invariant.
\end{proof}

\section{Cohomology of uniserial space groups and twisted abelian $p$-groups} \label{section:cohomologystandarduniserialspacegroup}

In this section, we study cohomology rings of quotients of uniserial space groups (see Subsections \ref{subsection:uniserialpadicspacegroupsandconstructible} and \ref{subsection:standarduniserialpadicspacegroups}) and cohomology rings of twisted abelian $p$-groups (see Subsection \ref{subsection:twistedabelianpgroups}). Let $\Z_p^{d_x}\rtimes W(x)$ denote the standard uniserial $p$-adic space group of dimensions $d_x=(p-1)p^{x-1}$. For $p$ odd, $W(x)$ is the unique maximal $p$-subgroup of $\GL_{d_x}(\Z_p)$, up to conjugation. Recall that for the $p=2$ case, there is another maximal $2$-subgroup (up to conjugacy) $\tilde{W}(x)$ in $\GL_{d_x}(\Q_2)$ described in Remark \ref{remark:sylowpeven}. J.F. Carlson shows that the cohomology of the quotients of $\Z_2^{d_x}\rtimes \tilde{W}(x)$ is determined (up to a finite number of possibilities) by the cohomology of the quotients of $\Z_2^{d_x}\rtimes W(x)$ (see \cite[p.259-260]{Carlson}).

More precisely, in the proof of \cite[Lemma 4.6]{Carlson}, it is shown that the cohomology of the quotients of $\Z_2^4\rtimes W(3)$ determines the cohomology of the quotients of $\Z_2^4\rtimes C_8$, where $C_8\leq Q_{16}=\tilde{W}(3)$ is the maximal subgroup of order $8$ in $Q_{16}$ and it is conjugate to an element of order $8$ in $W(3)$. Then, by the proof of \cite[Proposition 4.5]{Carlson}, the cohomology of the quotients of $\Z_2^4\rtimes C_8$ determines the cohomology of the quotients of $\Z_2^4\rtimes \tilde{W}(3)$. Finally, \cite[Proof of Proposition 4.7]{Carlson} proves that the cohomology of the quotients of $\Z_2^4\rtimes \tilde W(3)$ determines the cohomology of the quotients of $\Z_2^{d_x}\rtimes \tilde{W}(x)$ for $x\geq 4$. So, throughout the following subsections, we shall only state the results for point groups $P\leq W(x)$, although the same results hold for point groups $P\leq \tilde{W}(x)$.

\subsection{Cohomology of standard uniserial $p$-adic space group}\label{subsection:cohostandardpadicspacegroup}

Consider the standard uniserial space group of dimension $d_x=p^{x-1}(p-1)$, $\Z_p^{d_x}\rtimes W(x)$, defined in Subsection \ref{subsection:standarduniserialpadicspacegroups}.

\begin{Prop}\label{prop:standardspacegroupsamecohomology}
For fixed $x$, there are finitely many possibilities for the ring structure of the graded $\F_p$-module $H^*(\Z_p^{d_x}/U\rtimes W(x);\F_p)$ for the infinitely many $W(x)$-invariant lattices $U<\Z_p^{d_x}$.
\end{Prop}
\begin{proof}
\textbf{Step 1:} We first prove that there are finitely many ring structures for the infinite collection of  graded $\F_p$-modules $\{H^*(\Z_p^{d_x}/p^s\Z_p^{d_x} \rtimes W(x);\F_p)\}_{s\geq 1}$. Recall that, from Subsection \ref{subsection:standarduniserialpadicspacegroups} , the group $\Z_p^{d_x}\rtimes W(x)$ is the wreath product
\begin{equation}\label{equ:quotientbyUaswreath}
(\Z_p^{p-1}\rtimes C_p)\wr S \text{ with }S=\overbrace{C_p \wr \cdots \wr C_p}^{x-1},
\end{equation}
where the leftmost group $C_p$ is generated by the matrix $M$ of Equation \eqref{eq:action_maximal_class} and $S$ is Sylow $p$-subgroup of the symmetric group $\Sigma_{p^{x-1}}$. Now, let $s\geq 1$ and $s'\geq 1$. Then, $\Z_p^{d_x}/p^s\Z_p^{d_x}\rtimes W(x)$ and $\Z_p^{d_x}/p^{s'}\Z_p^{d_x}\rtimes W(x)$ may be written as the wreath products
\begin{equation}\label{equ:quotientbyUsaswreath}
(\Z_p^{p-1}/p^s\Z_p^{p-1}\rtimes C_p)\wr S
\text{ and }(\Z_p^{p-1}/p^{s'}\Z_p^{p-1}\rtimes C_p)\wr S
\end{equation}
respectively. Set $G=\Z_p^{p-1}/p^s\Z_p^{p-1}\rtimes C_p$ and $G'=\Z_p^{p-1}/p^{s'}\Z_p^{p-1}\rtimes C_p$. Then $G=K\rtimes C_p$ and $G'=K'\rtimes C_p$ where  $K=C_{p^s}\times\ldots\times C_{p^s}$ and $K'=C_{p^{s'}}\times\ldots\times C_{p^{s'}}$ are abelian groups of rank $p-1$. Moreover, the action of the generator of $C_p$ on $K$ and $K'$ is given by the matrix $M$. Hence, by Proposition \ref{prop:commonintegralliftingsamecohomology}, $H^*(G;\F_p)$ and $H^*(G';\F_p)$ are isomorphic $\F_p$-modules and there are finitely many possibilities for their ring structures when $s$ and $s'$ run over all integers greater or equal to $1$. Finally, by Nakaoka's theorem \cite[Theorem 5.3.1]{LEvens91}, there is an isomorphism of rings
\[
H^*(\Z_p^{d_x}/p^s\Z_p^{d_x}\rtimes W(x);\F_p)\cong H^*(S;{\stackrel{p^{x-1}}\otimes}H^*(G;\F_p)),
\]
where $S$ acts by permutations on the $p^{x-1}$ (tensor) copies of $H^*(G;\F_p)$. We conclude that there are finitely many possibilities for the ring structure of the graded $\F_p$-module $H^*(\Z_p^{d_x}/p^s\Z_p^{d_x}\rtimes W(x);\F_p)$ for the infinitely many choices of $s$.

\textbf{Step 2:} Now we consider the cohomology ring $H^*(\Z_p^{d_x}/U\rtimes W(x);\F_p)$, where $U$ is any $W(x)$-invariant sublattice of $\Z_p^{d_x}$. By uniseriality, there exists $s\geq 0$ such that $p^{s+1}\Z_p^{d_x}\leq U\leq p^s\Z_p^{d_x}$. In fact, leaving out a finite number of invariant sublattices we may assume that $s\geq 1$. Consider the short exact sequence of groups:
\[
0\to p^s\Z_p^{d_x}/U\to\Z_p^{d_x}/U\rtimes W(x) \to\Z_p^{d_x}/p^s\Z_p^{d_x}\rtimes W(x)\to 1.
\]
We have that $|p^s\Z_p^{d_x}/U|<p^{d_x}$, that $\rk(\Z_p^{d_x}/U\rtimes W(x))\leq d_x+\rk(W(x))$ and that, by Step $1$ above, there are finitely many ring structures for the quotient group $\Z_p^{d_x}/p^s\Z_p^{d_x}\rtimes W(x)$. Then by Theorem  \ref{thm:cohringsfromquotient} there are finitely many ring structures for $H^*(\Z_p^{d_x}/U\rtimes W(x);\F_p)$.
\end{proof}

\subsection{Cohomology of uniserial $p$-adic space groups}\label{subsection:cohopadicspacegroups}
Let $R$ be a uniserial $p$-adic space group of dimension $d_x=p^{x-1}(p-1)$ and coclass at most $c$ ($x\leq c$) and denote by $T$ its translation group and by $P$ its point group. We  define $T_0$ as the minimal $P$-lattice  for which the extension $T_0\to R_0\to P$ splits (see Subsection \ref{subsection:uniserialpadicspacegroupsandconstructible}). 



\begin{Prop}\label{prop:spacegroupssamebigradedalgebras}
Let $p$ be an odd prime. For each $s\geq 1$, there is a filtration of the ring  $H^*(T_0/p^sT_0\rtimes P;\F_p)$ such that the associated bigraded algebras are isomorphic for all $s$.
\end{Prop}
\begin{proof}
By Equation \eqref{equ:R_0containedZW(x)}, $T_0\rtimes P$ is a subgroup of the standard uniserial $p$-adic space group of dimension $d_x$, $\Z_p^{d_x}\rtimes W(x)$. The quotient $T_0/ p^s T_0$ is the $p^{x-1}$-fold direct product $K\times \ldots \times K$, where $K$ is the $(p-1)$-fold direct  product $K=C_{p^s}\times \ldots\times C_{p^s}$. As $P\leq W(x)$, the action of $P$ on $T_0$ is given by integral matrices and there is a common integral lifting that does not depend on $s$. Then the result follows from Proposition \ref{prop:commonintegralliftingsamecohomologyunbounded}\eqref{prop:commonintegralliftingsamecohomologyunboundedbigradedalgebras}.
\end{proof}


\begin{Prop}\label{prop:spacegroupssamecohomology}
There are finitely many possibilities for the ring structure of the graded $\F_p$-module $H^*(T_0/U\rtimes P;\F_p)$ for the infinitely many $P$-invariant lattices $U<T$. 
\end{Prop}
\begin{proof}
By Equation \eqref{equ:R_0containedZW(x)}, $T_0\rtimes P$ is a subgroup of the standard uniserial $p$-adic space group of dimension $d_x$, $\Z_p^{d_x}\rtimes W(x)$. We give two proofs. 

\textbf{Proof 1:}
By uniseriality, the $P$-sublattice $U$ is also an $W(x)$-sublattice. Then we have $T_0/U\rtimes P\leq \Z_p^{d_x}/U\rtimes W(x)$. By Proposition \ref{prop:standardspacegroupsamecohomology}, there are finitely many cohomology rings $H^*(\Z_p^{d_x}/U\rtimes W(x))$ when $U$ runs over the infinitely many $P$-sublattices. The index $|\Z_p^{d_x}/U\rtimes W(x):T_0/U\rtimes P|$ equals $|\Z_p^{d_x}:T_0||W(x):P|$ and does not depend on $U$. So by Theorem \ref{Car1}, there are finitely many ring structures for $H^*(T_0/U\rtimes P;\F_p)$.

\textbf{Proof 2:} By uniseriality, there exists $s\geq 0$ such that $p^{s+1}T_0\leq U\leq p^s T_0$ and we may assume that $s\geq 1$. Consider the short exact sequence of groups:
\[
0\to p^sT_0/U\to T_0/U\rtimes P \to T_0/p^s T_0\rtimes P\to 1.
\]
As $|p^s T_0/U|<pd_x$ and $\rk(T_0/U\rtimes P)\leq d_x+\rk(P)$, Theorem \ref{thm:cohringsfromquotient} shows that it is enough to prove that there are finitely many ring structures for the quotient group $T_0/p^s T_0\rtimes P$. This follows from the previous Proposition  \ref{prop:spacegroupssamebigradedalgebras} and Theorem \ref{Car2}.
\end{proof}

\begin{Cor}\label{cor:spacegroupssamecohomologynonsplit}
There are finitely many possibilities for the ring structure of the graded $\F_p$-module $H^*(R/U;\F_p)$ for the infinitely many $P$-invariant lattices $U<T$. 
\end{Cor}
\begin{proof}
We have $R\leq R_0=T_0\rtimes P$, $R/U\leq T_0/U\rtimes P$ and $|R_0:R|=|T_0:T|$. The result follows from Proposition \ref{prop:spacegroupssamecohomology} and Theorem \ref{Car1}.
\end{proof}

\subsection{Cohomology of twisted abelian $p$-groups.}\label{subsection:cohotwistedabelianpgroup}
Let $A$ be an abelian $p$-group for an odd prime $p$, let $P$ a $p$-group that acts on $A$ and let $\lambda \in \Hom_P(\Lambda^2 A,A)$. Consider then the group $A_\lambda$ constructed in Definition \ref{defi:Alambda}. Under the assumption that $A_\lambda$ is powerful $p$-central and has the $\Omega$EP, Theorem \ref{ThmWeigel}\eqref{ThmWeigelHG}  shows at once that $A$ and $A_\lambda$ have isomorphic cohomology rings. We show below how to compare the cohomology rings of $A\rtimes P$ and $A_\lambda\rtimes P$ (see Definition \ref{defi:AlambdasdP}). Before that, we give a precise statement of the Conjecture in the Introduction.

\begin{Conj}[{detailed statement of Conjecture \ref{conj:intro}}]\label{conj:detailed}
Let $A$ be an abelian $p$-group for an odd prime $p$, let $P$ a $p$-group and $P\to \Aut(A)$ an action with an integral lifting and let $\lambda \in \Hom_P(\Lambda^2 A,A)$. Assume that $A_\lambda$ is powerful $p$-central and has the $\Omega$EP. Then there is a zig-zag of quasi-isomorphisms in the category of cochain complexes,
\[
C^*(A;\F_p)\leftarrow U_1\rightarrow U_2\leftarrow\ldots \leftarrow U_r\rightarrow C^*(A_\lambda;\F_p),
\]
where each cochain complex $U_i$ has a product and a $P$-action, and each morphism is $P$-invariant and induces a ring isomorphism in cohomology.
\end{Conj}

Here, $C^*(A;\F_p)$ and $C^*(A_\lambda;\F_p)$ are the standard cochain complexes for $A$ and $A_\lambda$ and they already have products (see Subsection \ref{subsection:standardresolution}). In fact, $C^*(A;\F_p)$ could be replaced by $\Hom_{kA}(A_*,M)$ for $A_*$ any $\F_pA$-projective resolution (and similarly for $A_\lambda$). As in Subsection \ref{subsection:cohomologysmallrank}, we do not assume that the morphisms in the zig-zag commute with the products.

\begin{Prop}\label{prop:isoAAlambdartimesPrealized}
Assume that Conjecture \ref{conj:detailed} holds. Then there are filtrations of the graded algebras $H^*(A_\lambda\rtimes P;\F_p)$ and $H^*(A\rtimes P;\F_p)$ such that the associated bigraded algebras are isomorphic. 
\end{Prop}
\begin{proof}
The argument is similar to that in the proof of Proposition \ref{prop:commonintegralliftingsamecohomology}. Applying $\tot(\Hom_{\F_pP}(B_*(P;\F_p),-))$ to the zig-zag in Conjecture \ref{conj:detailed} we get a zig-zag of morphisms of cochain complexes:
\[
\tot(D_C)\leftarrow \tot(D_{U_1})\rightarrow\ldots\leftarrow\tot(D_{U_r})\rightarrow \tot(D_{C'}).
\]
Here $D_C=\Hom_{\F_pP}(B_*(P;\F_p),C^*(A;\F_p))$, $D_{C'}$ is defined analogously and $D_{U_i}=\Hom_{\F_pP}(B_*(P;\F_p),U_i)$. Then Lemma \ref{lemma:homalg} gives that $H^*(\tot(D_C))$, $H^*(\tot(D_{C'}))$ and $H^*(\tot(D_{U_i}))$ are all isomorphic graded $\F_p$-vector spaces. Using now the products on $C^*(A;\F_p)$, $C^*(A_\lambda;\F_p)$ and each $U_i$, the Proposition follows from Lemma \ref{lemma:homalgalgebra}.
\end{proof}

\section{Carlson's conjecture}
\label{section:carlsonconjecture}
Now we are ready to prove the main result of this work. The non-twisted case ($\gamma=0$, see Remark \ref{rmk:nontwistedcase}) follows verbatim from the same proof below and without the assumption that Conjecture \ref{conj:detailed} holds.

\begin{Thm}\label{thm:Carslonconj}
Let $p$ be an odd prime and let $c$ be an integer. If Conjecture \ref{conj:detailed} holds, then there are finitely many ring isomorphism types for the cohomology rings $H^*(G;\F_p)$ when $G$ runs over the set of all finite $p$-groups of coclass $c$.
\end{Thm}

\begin{proof}
Let $G$ be a finite $p$-group of coclass $c$. By Theorem \ref{thm:quotientofcoclasscareconstructible}, there exists an integer $f(p,c)$ and a normal subgroup $N$ of $G$ with $|N|\leq f(p,c)$ such that $G/N$ is constructible. A constructible group $G_\gamma$ arises from a quadruple $R$, $U$, $V$ and $\gamma$ (see Definition \ref{Defi:constructiblegroup}). Here, $R$ is a uniserial $p$-adic space group of dimension $d_x=p^{x-1}(p-1)$ and coclass at most $c$ ($x\leq c$), and there are finitely many such groups. We denote by $T$ the translation group of $R$ and by $P$ its point group. We also define $T_0$ as the minimal $P$-lattice  for which the extension $T_0\to R_0\to P$ splits (see Subsection \ref{subsection:uniserialpadicspacegroupsandconstructible}). So we have $G/N\cong G_\gamma$ and, by Lemma \ref{lemma:GalphaGalpha0}\eqref{lemma:GalphaGalpha0subgroupindex}, $G_\gamma$ is a subgroup of $G_{\gamma,0}$ (see Definition \ref{Defi:splitconstructiblegroup}). Then, by Lemma \ref{lemma:GalphaGalpha0}\eqref{lemma:GalphaGalpha0split}, we have $G_{\gamma,0}\cong (T_0/U,+_{\lambda})\rtimes P$.

We shall show that
\begin{equation}\label{claim:mainthm}
\text{there are finitely many ring structures for $H^*(G_{\gamma,0};\F_p)$}
\end{equation}
considering three different cases. Before doing so, notice that 
\begin{equation}\label{equ:mainproofrankisbounded}
\rk((T_0/U,+_{\lambda})\rtimes P)\text{ is bounded}
\end{equation}
when $G_\gamma$ runs over the collection of constructible groups above. This is a consequence of the inequality 
\begin{align*}
\rk((T_0/U,+_{\lambda})\rtimes P)&\leq \rk(T_0/U,+_{\lambda})+\rk(P)\\
&\leq \rk(T_0/V, +_{\lambda})+\rk(V/U,+_{\lambda})+\rk(P)\\
&\leq \rk(T_0/V, +)+\rk(V/U,+)+\rk(P)\\
&\leq 2d_x+\rk(P).
\end{align*}
and the fact that there finitely many point groups $P$ (see also  Lemma \ref{propertiesAlambda}\eqref{propertiesAlambdasameOmega1} for the properties of $+_{\lambda}$).

\textbf{Case 1: $\frac{1}{p^3}U\leq V\leq pT_0$:} This implies that, in particular, hypothesis \eqref{equation:extraassumptionsUV} holds. Let $i\geq 2$ be the only integer satisfying $p^{i+2}T_0< U\leq p^{i+1}T_0$. We have the following short exact sequence
\begin{equation}\label{equ:finalproofhypothesisholdses}
(p^iT_0/U,+_{\lambda})\longrightarrow (T_0/U,+_{\lambda})\rtimes P\longrightarrow (T_0/p^iT_0,+_{\lambda'})\rtimes P,
\end{equation}
where, by Lemma \ref{propertiesT0U+*}\eqref{propertiesT0U+*TOWOmegaEP} with $W=p^iT_0$, $(T_0/p^iT_0,+_{\lambda'})$ is powerful $p$-central group with $\Omega$EP. By Proposition \ref{prop:isoAAlambdartimesPrealized}, the cohomology rings $H^*((T_0/p^iT_0,+_{\lambda'})\rtimes P;\F_p)$ and $H^*(T_0/p^iT_0\rtimes P;\F_p)$ have filtrations with isomorphic associated bigraded algebras. By Proposition \ref{prop:spacegroupssamebigradedalgebras}, this bigraded algebra does not depend on $i$. Hence, there are finitely many cohomology rings for $H^*((T_0/p^iT_0,+_{\lambda'})\rtimes P;\F_p)$ for the infinitely many $i$'s and $\gamma$'s.  In the short exact sequence \eqref{equ:finalproofhypothesisholdses}, we have:
\begin{enumerate}[({Case 1.}a)]
\item the kernel has size at most $p^{2d_x}$,
\item $\rk((T_0/U,+_{\lambda})\rtimes P)$ is bounded by Equation \eqref{equ:mainproofrankisbounded}, and
\item $(T_0/p^iT_0,+_{\lambda'})$ has nilpotency class $2$ by Lemma \ref{propertiesAlambda}\eqref{propertiesAlambdaclasstwo}.
\end{enumerate}
By Theorem \ref{thm:cohringsfromquotient}, there are finitely many choices for the ring $H^*((T_0/U,+_{\lambda})\rtimes P;\F_p)$.

\textbf{Case 2: $pT_0\leq V$:} If $U\geq pT_0$ then we are leaving out a finite number of cases. Otherwise $U\leq pT_0$ and we proceed as follows: Note first that
\[
p\cdot \gamma(T_0/V,T_0/V)=\gamma(pT_0/V,T_0/V)\leq \gamma(V/V,T_0/V)=0.
\]
Then, $\im(\lambda)\subseteq (\frac{1}{p}U)/U$. Let $i\geq 1$ denote the largest integer such that $\frac{1}{p}U\leq p^iT_0$. In the short exact sequence
\[
(p^iT_0/U,+_{\lambda})\longrightarrow (T_0/U,+_{\lambda})\rtimes P\longrightarrow T_0/p^iT_0\rtimes P,
\]
we have $(T_0/p^iT_0,+)=(T_0/p^iT_0,+_\lambda)$ as $\im(\lambda)\subseteq (\frac{1}{p}U)/U \subseteq p^iT_0/U$, and:
\begin{enumerate}[({Case 2.}a)]
\item the kernel has size at most $p^{2d_x}$,
\item $\rk((T_0/U,+_{\lambda})\rtimes P)$ is bounded by Equation \eqref{equ:mainproofrankisbounded}, and
\item $T_0/p^iT_0$ is an abelian subgroup of finite index of $T_0/p^iT_0\rtimes P$. 
\end{enumerate}
By Proposition \ref{prop:spacegroupssamebigradedalgebras}, there are finitely many ring structures for $H^*(T_0/p^iT_0\rtimes P;\F_p)$ and by Theorem \ref{thm:cohringsfromquotient}, there are finitely many choices for the ring $H^*((T_0/U,+_{\lambda})\rtimes P;\F_p)$.


\textbf{Case 3: $p^3V\leq U$:} As $U\leq V$ we have $|V:U|\leq |V:p^3V|\leq p^{3d_x}$. We also have a short exact sequence
\begin{equation}
\label{equ:finalproofcase3ses}
V/U \longrightarrow (T_0/U,+_{\lambda})\rtimes P\longrightarrow T_0/V\rtimes P,
\end{equation}
where the description of the rightmost group is a consequence of Lemma \ref{propertiesAlambda}\eqref{propertiesAlambdaisoonquotient}. By Proposition \ref{prop:spacegroupssamecohomology}, there are finitely many ring structures for $H^*(T_0/V\rtimes P;\F_p)$. Moreover, in the short exact sequence \eqref{equ:finalproofcase3ses}, we have:
\begin{enumerate}[({Case 3.}a)]
\item the kernel has size bounded by $p^{3d_x}$, 
\item $\rk((T_0/U,+_{\lambda})\rtimes P)$ is bounded by Equation \eqref{equ:mainproofrankisbounded}, and
\item $T_0/V$ has nilpotency class $1$ (abelian).
\end{enumerate}
By Theorem \ref{thm:cohringsfromquotient}, there are finitely many choices for the ring  $H^*((T_0/U,+_{\lambda})\rtimes P;\F_p)$.

So \eqref{claim:mainthm} is proven. Now, by Lemma \ref{lemma:GalphaGalpha0}\eqref{lemma:GalphaGalpha0subgroupindex}, the index of $G_\gamma$ in $G_{\gamma,0}$ is exactly $|T_0:T|$, and this is again bounded when $R$ runs over the finitely many uniserial $p$-adic space groups of coclass $\leq c$. So, by Theorem \ref{Car1}, there are finitely many ring structures for $H^*(G_{\gamma};\F_p)$. Thus, $G$ fits in an extension of groups,
\[
1 \to N \to G \to G_{\gamma} \to 1,
\]
where:
\begin{enumerate}[(a)]
\item the kernel has size $|N|\leq f(p,c)$,
\item $\rk(G)\leq \rk(N)+\rk(G_\gamma)\leq |N|+\rk(G_{\gamma,0})\leq f(p,c)+\rk((T_0/U,+_{\lambda})\rtimes P)$ is bounded by Equation \eqref{equ:mainproofrankisbounded}, where $G_\gamma\leq G_{\gamma,0}\cong (T_0/U,+_{\lambda})\rtimes P$, and
\item $G_\gamma\cap (T_0/U,+_{\lambda})$ is a nilpotency class $2$ (normal) subgroup of $G_\gamma$ by Lemma \ref{propertiesAlambda}\eqref{propertiesAlambdaclasstwo}, and it has index bounded by $|P|$.
\end{enumerate}
Then, by Theorem \ref{thm:cohringsfromquotient}, there are finitely many ring structures for $H^*(G;\F_p)$. We are done.
\end{proof}


\end{document}